\documentclass[12pt]{amsart}
\allowdisplaybreaks[1]

\usepackage{color}

\usepackage{chngcntr}

\usepackage{amssymb,amsmath,amsfonts,amsthm,nomencl,mathrsfs}
\usepackage[latin1]{inputenc}
\newcommand{\IR}{\mathbb{R}}
\newcommand{\grad}{\mathrm{grad}}

\newcommand{\IHH}{\mathscr{H}}

\newcommand{\z}{ \Id \mu }

\newcommand{\IN}{\mathbb{N}}

\newcommand{\Id}{{\rm d}}

\newcommand{\f}{\frac}
\newcommand{\nn}{\nonumber}

\newcommand{\Ic}{\mathrm{c}}
\newcommand{\T}{\mathrm{T}}
\newcommand{\IL}{\mathsf{L}}
\newcommand{\C}{\mathsf{C}}

\newtheorem{theorem}{THEOREM}[section]

\newtheorem{Lemma}[theorem]{Lemma}

\newtheorem{Example}[theorem]{Example}

\newtheorem{Corollary}[theorem]{Corollary}

\newtheorem{Remark}[theorem]{Remark}

\newtheorem{Theorem}[theorem]{Theorem}

\newtheorem{Proposition}[theorem]{Proposition}

\newtheorem{Definition}[theorem]{Definition}

\newtheorem*{Acknowledgments}{Acknowledgments}
\newcommand{\question}[1]{\leavevmode{\marginpar{\tiny
$\hbox to 0mm{\hspace*{-0.5mm}$\leftarrow$\hss}
\vcenter{\vrule depth 0.1mm height 0.1mm width \the\marginparwidth}
\hbox to 0mm{\hss$\rightarrow$\hspace*{-0.5mm}}$\\\relax\raggedright #1}}}

\begin{document}

\title[ Calder\'on-Zygmund inequality on Riemannian manifolds]{The Calder\'on-Zygmund inequality and Sobolev spaces on noncompact Riemannian manifolds}

\author[B. G\"uneysu]{Batu G\"uneysu}
\address{Batu G\"uneysu, Institut f\"ur Mathematik, Humboldt-Universit\"at zu Berlin, 12489 Berlin, Germany} \email{gueneysu@math.hu-berlin.de}

\author[S. Pigola]{Stefano Pigola}
\address{Stefano Pigola, Dipartimento di Scienza e Alta Tecnologia - Sezione di Matematica, Università dell'Insubria, 22100 Como, Italy} \email{stefano.pigola@uninsubria.it}


\maketitle 

\begin{abstract} We introduce the concept of Calder\'on-Zygmund inequalities on Riemannian manifolds. For $1<p<\infty$, these are inequalities of the form
$$
\left\Vert \mathrm{Hess}\left(  u\right)  \right\Vert _{\IL^p}\leq
C_{1}\left\Vert u\right\Vert _{\IL^p}+C_{2}\left\Vert \Delta u\right\Vert
_{\IL^p}, 
$$
valid a priori for all smooth functions $u$ with compact support, and constants $C_1\geq 0$, $C_2>0$. Such an inequality can hold or fail, depending on the underlying Riemannian geometry. After establishing some generally valid facts and consequences of the Calder\'on-Zygmund inequality (like new denseness results for second order $\mathsf{L}^p$-Sobolev spaces and gradient estimates), we establish sufficient geometric criteria for the validity of these inequalities on possibly noncompact Riemannian manifolds. These results in particular apply to many noncompact hypersurfaces of constant mean curvature.   
\end{abstract}

\section{Introduction}

Let $M$ be a smooth possibly noncompact Riemannian manifold. For an \emph{arbitrary} $p\in (1,\infty)$, let us consider the following canonically given problems for second order Sobolev spaces on $M$, on the $\mathsf{L}^{p}$-scale:
\begin{itemize}  
	\item \emph{Problem 1: Under which (geometric) assumptions on $M$ does one have the denseness $\mathsf{H}^{2,p}_0(M)=\mathsf{H}^{2,p}(M)$? 
\item Problem 2: Under which assumptions on $M$ does one have the implication
  $$
 f\in \mathsf{L}^{p}(M)\cap \mathsf{C}^{2}(M),\Delta f \in \mathsf{L}^{p}(M) \Rightarrow f\in \mathsf{H}^{2,p}(M)
  $$
(that is, $|\mathrm{Hess}(f)|\in \mathsf{L}^{p}(M)$)?
 \item Problem 3: Under which assumptions on $M$ does one have an inequality of the form}
  \begin{align}\label{intrograd}
\left\|  \grad(f) \right\|_{\mathsf{L}^{p}} \leq C (\left\|  \Delta f \right\|_{\mathsf{L}^{p}}+\left\|  f \right\|_{\mathsf{L}^{p}})\>\text { for all $ f\in \mathsf{C}^{\infty}_{\mathrm{c}}(M)$}?
  \end{align}
\end{itemize}

Let us note here that Problem 1, the denseness of $\mathsf{C}^{\infty}_{\mathrm{c}}(M)$ in $\mathsf{H}^{2,p}(M)$, is a classical problem which has been treated systematically in \cite{hebey}. Here, we would like to stress the fact that without a lower control on the injectivity results, nothing seems to be known so far for the case $p\ne 2$. Furthermore, Problem 2 is obviously concerned with $\mathsf{L}^{p}$-estimates for solutions of the Poisson equation on $M$, and Problem 3 arises naturally in Rellich-Kondrachov type compactness arguments: Here, typically one has given a sequence of functions $\{f_n\}\subset \mathsf{C}^{\infty}_{\mathrm{c}}(M)$ such that 
$$
\sup_{n\in\IN}\max (\left\|f_n\right\|, \left\|\Delta f_n\right\|) <\infty,
$$
and one would like to know whether the sequence $\{f_n\}$ is bounded in $\mathsf{H}^{1,p}(M)$.\\
It turns that there is an inequality underlying all three problems simultaniously, namely, the \emph{Calder\'on-Zygmund inequality}. This inequality, which may generally fail on noncompact $M's$, states that there are constants $C_1\geq 0$, $C_2>0$ such that
$$
(\mathrm{CZ}(p))\>\>\>    \>\>\>\>\>\>\left\Vert \mathrm{Hess}\left(  u\right)  \right\Vert _{\IL^p}\leq
C_{1}\left\Vert u\right\Vert _{\IL^p}+C_{2}\left\Vert \Delta u\right\Vert
_{\IL^p}\>\>\text{ for all $u\in  \mathsf{C}^{\infty}_{\mathrm{c}}(M)$.} 
$$
Let us remark that in the Euclidean $\IR^m$, this inequality can be proved \cite{GT} by estimates on singular integral operators which have been proved by Calder\'on and Zygmund \cite{cald}. Ultimately, this was the motivation for us for calling\footnote{The authors would like to thank Klaus Ecker in this context} $\mathrm{CZ}(p)$ the \lq\lq{}Calder\'on-Zygmund inequality\rq\rq{}.\\ 
Now one has the following implications which link $\mathrm{CZ}(p)$ to the above Problems 1-3:

\begin{itemize}  
	\item[(A)] It has been observed in \cite{Gu}  \emph{that under $\mathrm{CZ}(p)$, if $M$ is geodesically complete and admits a sequence of Laplacian cut-off functions (this is the case e.g. if $M$ has a nonnegative Ricci curvature), then one has $\mathsf{H}^{2,p}_0(M)=\mathsf{H}^{2,p}(M)$}.\vspace{2mm}

\item[(B)] In Corollary \ref{saa3} \emph{we show that under $\mathrm{CZ}(p)$, if $M$ is geodesically complete and admits a sequence of Hessian cut-off functions (this is the case e.g. if $M$ has a bounded curvature tensor), then any $f\in \mathsf{H}^{1,p}(M)\cap \mathsf{C}^{2}(M)$ with $\Delta f \in \mathsf{L}^{p}(M)$ satisfies $f\in \mathsf{H}^{2,p}(M)$.}\vspace{2mm}

\item[(C)] In Corollary \ref{cor_gradestimate} \emph{we prove that, on a geodesically complete manifold, $\mathrm{CZ}(p)$ always implies (\ref{intrograd})}.
\end{itemize}

Here, the results (A) and (B) follow from the existence of appropriate second order cut-off functions (see also \cite{Gu}), which we prove to exist under very weak assumptions on the curvature and without positive injecitivity radius. In this context, we also establish our first main result (cf. Corollary \ref{dssss} below):

\begin{Theorem} Let $M$ be geodesically complete with a bounded curvature tensor. Then one has $\mathsf{H}^{2,p}_0(M)= \mathsf{H}^{2,p}(M)$ for all $1<p<\infty$.
\end{Theorem}

This result is entirely new for $p\ne 2$, for it does not require a positive injectivity radius (cf. \cite{hebey}). \\
The statement (C) makes use of an appropriate $\mathsf{L}^p$-interpolation result, which should be of an independent interest (cf. Proposition \ref{lemma_interpolation}). These observations clearly motivate a systematic treatement of the following problem:\vspace{2mm}

\emph{Under which (geometric) assumptions on $M$ does one have $\mathrm{CZ}(p)$, and how do the $\mathrm{CZ}(p)$-constants $C_1$, $C_2$ depend on the underlying geometry?}\vspace{2mm}

Let us start by taking a look at the \emph{local situation}: We first prove \emph{in Theorem \ref{relco} that one always has $\mathrm{CZ}(p)$ on relatively compact domains $\Omega \subset M$, where one can even pick $C_1=0$, if $\Omega$ has a smooth boundary}. In particular, using a gluing procedure which again relies on $\mathsf{L}^p$-interpolation, the latter results show \emph{that $\mathrm{CZ}(p)$ is stable under compact perturbations (cf. Theorem \ref{relco})}, which in particular applies to manifolds with ends. However, as one might expect, in both of these cases the $\mathrm{CZ}(p)$-constants depend rather implicitely on the underlying geometry, which raises the question of more precise estimates on geodesic balls: This problem is attacked \emph{in Theorem \ref{local}, where we prove that $\mathrm{CZ}(p)$ holds on sufficiently small geodesic balls, with constants only depending on the radius, $\dim M$, $p$, and a lower bound of an appropriate harmonic radius. }\\
As for \emph{global results}, it turns out that for $p=2$ it is possible to give a rather complete answer: Namely, it is \emph{shown in Proposition \ref{Ricbelow} that a lower bound $\mathrm{Ric}\geq -C$ on the Ricci curvature implies a stronger infinitesimal variant of $\mathrm{CZ}(2)$ with constants depending explicitely on $C$.} This is in fact a straightforward consequence of Bochner\rq{}s indentity. On the other hand, \emph{we prove that this result is optimal, in the sense that there exists a geodesically complete noncompact surface $N$ with unbounded Gauss curvature, such that $\mathrm{CZ}(2)$ fails on $N$ (cf. Theorem \ref{counter}).} \\
For $p\ne 2$, we prove the following two results, which can also be considered as the main results of this paper:

\begin{Theorem}\label{centralei} Let $1<p<\infty$ and assume that $M$ has bounded Ricci curvature and a positive injectivity radius\footnote{thus $M$ is automatically geodesically complete}. Then one has $\mathrm{CZ}(p)$, with constants depending only on $\dim M$, $p$, $\left\|\mathrm{Ric}\right\|_{\infty}$ and the injectivity radius.
\end{Theorem}

\begin{Theorem}\label{centraleii} Let $1<p\leq 2$. Assume that $M$ is geodesically complete with a first order bounded geometry, and that there are constants $D\geq 1$, $0\leq\delta<2$ with 
\begin{align}
\mathrm{vol}(\mathrm{B}_{tr}(x))\leq D t^{D}\mathrm{e}^{t^{\delta
}+r^{\delta}}\mathrm{vol}(\mathrm{B}_r(x))\>\text{ for all $x\in M$, $r>0$, $t\geq1$.}\label{B111}
\end{align}
Then one has $\mathrm{CZ}(p)$, with constants depending only on $\dim M$, $p$, $\left\|\mathrm{R}\right\|_{\infty}$, $\left\|\nabla \mathrm{R}\right\|_{\infty}$, $D$,$\delta$, with $\mathrm{R}$ the curvature tensor.
 \end{Theorem}

The proof of Theorem \ref{centralei} uses appropriate elliptic estimates, in combination with harmonic-radius bounds on the Riemannian structure, a method which requires bounds on the injectivity radius, but has the advantage of working for arbitrary $p$. The proof of Theorem \ref{centraleii}, however, is very different: It uses deep boundedness-results on \emph{covariant Riesz-transforms} by Thalmaier-Wang \cite{thalmaier}, which ultimately follow from covariant probabilistic heat-semigroup derivative formulas. This technique makes it possible to avoid assumptions on the injectivity radius. Under geodesic completeness, the generalized doubling assumption (\ref{B111}) is implied by $\mathrm{Ric}\geq 0$ (though $\mathrm{Ric}\geq 0$ is not necessary at for all (\ref{B111}); cf. Example \ref{negric}).

This paper is organized as follows: In Section \ref{anfang} we establish some Riemann geometric notation. In Section \ref{cons} we first establish the above mentioned consequences (A), (B), (C) of the Calder\'on-Zygmund inequality, and then we prove the various local Calder\'on-Zygmund inequalities (Theorem \ref{relco} and Theorem \ref{local}). In Section \ref{gross} we prove several geometric criteria for the validity $\mathrm{CZ}(p)$ on noncompact $M\rq{}s$, thus Proposition \ref{Ricbelow} for $p=2$, and the above Theorem \ref{centralei} and Theorem \ref{centraleii}, and Section \ref{model} is devoted to the construction of the surface which does not support $\mathrm{CZ}(2)$ fails. Finally, in Section \ref{etre2} we apply Theorem \ref{centralei}  to noncompact hypersurfaces of constant mean curvature. We have also included two appendices, where some facts on harmonic-radius bounds and on abstract Riemannian gluings have been collected for the convenience of the reader.

\section{Setting and notation}\label{anfang}

We fix an arbitrary smooth Riemannian $m$-manifold  $M\equiv (M,g)$ \footnote{In the sequel a manifold will always be understood to be without boundary, unless otherwise stated} with $\nabla$ the Levi-Civita connection on $M$. We will denote the corresponding distance function with $\Id(\bullet,\bullet)$, the open balls with $\mathrm{B}_a(x)$, $x\in M$, $a>0$, and the volume measure with $\mu(\Id x):=\mathrm{vol}(\Id x)$, where whenever there is no danger of confusion we shall simply write $\int f \z$ instead of $\int_M f\z$. The symbol $r_{\mathrm{inj}}(x)\in (0,\infty]$ will stand for the injectivity radius at $x$, with 
$$
r_{\mathrm{inj}}(M):=\inf_{x\in M}r_{\mathrm{inj}}(x)\in [0,\infty]
$$
the global injectivity radius.\\
Let us develop some further geometric notation which will be used in the sequel: If $E\to M$ is a smooth Euclidean vector bundle, then whenever there is no danger of confusion we will denote the underlying Euclidean structure simply with $(\bullet,\bullet)_x$, $x\in M$, $\left|\bullet\right|_x:=\sqrt{(\bullet,\bullet)}_x$ will stand for the corresponding norm on $E_x$. Using $\mu$, for any $1\leq p\leq \infty$ we get the corresponding $\IL^p$-spaces of equivalence classes of Borel sections $\Gamma_{\IL^p}(M,E)$, with their norms
$$
\left\| f\right\|_{p}:=\begin{cases}&\Big(\int_M \big|f(x)\big|^p_x  \mu(\Id x)\Big)^{1/p} 
,\text{ if $p<\infty$} \\
&\inf\{C|C\geq 0, |f|\leq C\text{ $\mu$-a.e.}\},\text{ if $p=\infty$.}\end{cases}
$$
The symbol $\left\langle\bullet,\bullet \right\rangle$ will stand for the canonical inner product on the Hilbert space $\Gamma_{\IL^2}(M,E)$, and  '$\dagger$' will denote the formal adjoint with respect to $\left\langle\bullet,\bullet \right\rangle$ of a smooth linear partial differential operator that acts on some $E\to M$ as above.\\
We equip $\T^*M$ with its canonical Euclidean structure 
$$
(\alpha_1,\alpha_2):= g(\alpha^{\sharp}_1,\alpha^{\sharp}_2),  \>\> \alpha_1,\alpha_2\in\Gamma_{\C^{\infty}}(M,\T^*M),
$$
where $\alpha^{\sharp}_j$ stands for the vector field which is defined by $\alpha$ in terms of $g$. This produces canonical Euclidean metrics on all bundles of $k$-times contravariant and $l$-times covariant tensors $\T^{k,l}M\to M$. Next, $\nabla$ induces a Euclidean covariant derivative on $\T^*M=\T^{0,1}\to M$ through
$$
\nabla_{X_1} \alpha(X_2) := X_1(\alpha(X_2))-\alpha(\nabla_{X_1}X_2),  
$$
for any smooth $1$-form $\alpha$ and any smooth vector fields  $X_1,X_2$ on $M$, which of course means nothing but $(\nabla_{X_1}\alpha)^{\sharp}= \nabla_{X_1}\alpha^{\sharp}$, and these data are tensored to give a Euclidean covariant derivative on $\T^{k,l}M\to M$ which, by the above abuse of notation, is always denoted with $\nabla$. Thus, for any $u\in \C^{\infty}(M)$ we have 
\begin{align*}
\> \mathrm{Hess}(u):=\nabla\Id u \in \Gamma_{\C^{\infty}}(M,\T^{0,2} ).
\end{align*}
The gradient $\grad(u)\in \Gamma_{\C^{\infty}}(M,\T M)$, is defined by 
$$
(\grad(u), X):=\Id u (X)\text{ for any $X\in\Gamma_{\C^{\infty}}(M,\T M)$},
$$ 
with
$$
\Id:\Gamma_{\C^{\infty}}(M,\wedge^{\bullet} \T^* M)\longrightarrow \Gamma_{\C^{\infty}}(M,\wedge^{\bullet+1}\T^* M )
$$
the exterior differential, and where as usual $0$-forms are identified with functions. Then the divergence $\mathrm{div}(X)\in\C^{\infty}(M)$ of a smooth vector field $X$ on $M$ is given by $\mathrm{div}(X)=\Id^{\dagger} X^{\flat}$, where $X^{\flat}$ stands for the $1$-form which is defined by $X$ in terms of $g=(\bullet,\bullet)$. Let us denote with 
$$
\Delta_{\bullet}:=\Id^{\dagger}\Id + \Id \Id^{\dagger}:\Gamma_{\C^{\infty}}(M,\wedge^{\bullet}\T^* M )\longrightarrow \Gamma_{\C^{\infty}}(M,\wedge^{\bullet}\T^* M )
$$
the Laplace-Beltrami operator on differential forms. Note that our sign convention is such that \emph{$\Delta_{\bullet}$ is nonnegative}, and the Friedrichs realization of $\Delta_{\bullet}$ in $\Gamma_{\mathsf{L}^{2}}(M,\wedge^{\bullet}\T^* M )$ will be denoted with the same symbol. In the sequel, we will freely use the formulas
\begin{align*}
&\mathrm{Hess}(u_1 u_2) =u_2\mathrm{Hess}(u_1) + \Id u_1\otimes \Id u_2+\Id u_2\otimes \Id u_1+ u_1\mathrm{Hess}(u_2),\\
 &\Delta u_1=-\mathrm{tr}(\mathrm{Hess}(u_1))=-\mathrm{div}(\grad(u_1)),\\
&\Delta(u_1 u_2)=  \Delta(u_1) u_2- 2(\grad(u_1),\grad(u_2))+\Delta(u_2) u_1, \\
&\mathrm{div}(u_1 X) =(\grad(u_1),X)+ u_1 \mathrm{div}(X),
\end{align*}
valid for all smooth functions $u_1$, $u_2$ and smooth vector fields $X$ on $M$. \\
We close this section with some conventions and notation which concerns curvature data: The curvature tensor $\mathrm{R}$ is read as a section 
$$
\mathrm{R}\in \Gamma_{\C^{\infty}}(M, \T^{1,3} M ),
$$
given for smooth vector fields $X,Y,Z$ by 
$$
\mathrm{R}(X, Y,Z):= \nabla_X \nabla_Y Z-\nabla_Y\nabla_X Z-\nabla_{[X,Y]}Z\in \Gamma_{\C^{\infty}}(M,  \T M).
$$
Then the Ricci curvature 
$$
\mathrm{Ric}\in \Gamma_{\C^{\infty}}(M,  \T^{0,2} M)
$$
is the section given by the fiberwise trace 
$$
\mathrm{Ric}(Z,Y)=\mathrm{tr}(X\mapsto \mathrm{R}(X, Y,Z))\in \C^{\infty}(M).
$$
If $m\geq 2$, then for any $x\in M$, the sectional curvature of a two dimensional subspace $A=\mathrm{span}(X,Y)$ of $\T_x M$ is well-defined by
$$
\mathrm{Sec}(A):=\f{(\mathrm{R}(X,Y,Y),X)}{|X\wedge Y|^2}\in\IR.
$$
Finally, we mention that whenever we write $C=C(a_1,\dots,a_l)$ for a constant, this means that $C$ only depends on the parameters $a_1,\dots,a_l$, and nothing else. 

\section{Consequences of the Calder\'on-Zygmund inequality and local considerations}\label{cons}

In this section, we are going to collect some abstract and fundamental facts on Calder\'on-Zygmund inequalities. \\

We start with:

\begin{Definition}\label{deff} Let $1<p<\infty $. We say that $M\equiv (M,g)$ satisfies the $\mathsf{L}^p$-Calder\'on-Zygmund inequality (or in short $\mathrm{CZ}(p)$), if there are $C_1\geq 0$, $C_2>0$, such that for all $u\in\C^{\infty}_{\Ic}(M)$ one has
 \begin{equation}
\left\Vert \mathrm{Hess}\left(  u\right)  \right\Vert _{p}\leq
C_{1}\left\Vert u\right\Vert _{p}+C_{2}\left\Vert \Delta u\right\Vert
_{p}. \label{cz}
\end{equation}
\end{Definition}
Obviously, if some $\mathrm{CZ}(p)$ holds on a Riemannian manifold $M$ then, by restriction, the same inequality holds on any open subset of $M$ and with the same constants. Furthermore, we have directly excluded the extremal cases $p=1$ and $p=\infty$ in Definition \ref{deff} since the corresponding hypothetical elliptic estimates fail \cite{ornstein, delee} for the Euclidean Laplace operator $\sum_i \partial_i^2$, thus (\ref{cz}) with $p=1$ or $p=\infty$ cannot hold in general. \\
Let us record that a certain scaling rigidity in the constants implies automatically that one can pick $C_1=0$ (noting that in Riemannian geometry, such a stability typically appears in the context of nonnegative Ricci curvature): 

\begin{Remark} 1. Let $1<p<\infty$, and assume that there are $C_1\geq 0 $, $C_2>0$ such that for all $0<\lambda\leq 1$ one has $(\ref{cz})$ with respect to the Riemannian metric $\lambda^2g$. Then one has $(\ref{cz})$ with $C_1=0$ (with respect to $g$). Indeed, the assumption implies
\[
\left\Vert \mathrm{Hess}\left(  u\right)  \right\Vert _{p}\leq
C_{1}\lambda^{2p}\left\Vert u\right\Vert _{p}+C_{2}\left\Vert \Delta
u\right\Vert _{p},
\]
with respect to $g$ with $C_j$ uniform in $\lambda$, and we can take $\lambda\rightarrow 0+$. \\
2. As a particular case of the above situation, assume $\mathrm{Ric}\geq 0$, $1<p<\infty$ and that there are $C_1\geq 0 $, $C_2>0$, which do not depend on $g$, such that one has $(\ref{cz})$. Then one has $(\ref{cz})$ with $C_1=0$.
\end{Remark}

We continue with some important consequences of the Calder\'on-Zygmund inequalities. Let us start with some remarks concerning the connection between $\mathrm{CZ}(p)$ and second order $\mathsf{L}^p$-Sobolev spaces. In fact, precisely this context was the original motivation for our study of Calder\'on-Zygmund inequalites. To this end, we first list some conventions and notation on Sobolev spaces: For any $1\leq  p<\infty$, the Banach space $\mathsf{H}^{1,p}(M)$ is defined by 
$$
\mathsf{H}^{1,p}(M):=\left\{u\left| \>u\in\mathsf{L}^{p}(M), |\mathrm{grad}(u)|\in\mathsf{L}^{p}(M)\text{ as distr.}\right\}\right.,
$$
with its natural norm 
\begin{align*}
\left\| u\right\|_{1,p}:=  \left\|u \right\|_p + \left\| \mathrm{grad}(u) \right\|_p.
\end{align*}
Likewise, one has the Banach space
$$
\mathsf{H}^{2,p}(M):=\left\{u\left| \>u\in\mathsf{L}^{p}(M), |\mathrm{grad}(u)|,|\mathrm{Hess}(u)|\in\mathsf{L}^{p}(M)\text{ as distr.}\right\}\right.,
$$
with its natural norm $\left\| u\right\|_{2,p}$. By a generalized Meyers-Serrin type theorem \cite{guidetti}, one has that the linear space
$$
\text{$\mathsf{C}^{\infty}(M)\cap \mathsf{H}^{k,p}(M)$ is dense in $\mathsf{H}^{k,p}(M)$}
$$
(a fact which is actually true for all $k\in\IN$ with the natural definition of higher order Sobolev spaces). Finally, we define $\mathsf{H}^{k,p}_0(M)\subset \mathsf{H}^{k,p}(M)$ as usual to be the closure of $\C^{\infty}_{\Ic}(M)$ in $\mathsf{H}^{k,p}(M)$. 

\begin{Remark}\label{sobo} Let $1\leq  p<\infty$. \\
1. Every $u\in \mathsf{H}^{2,p}(M)$ satisfies $\Delta u\in \mathsf{L}^p(M)$ in the sense of distributions. Indeed, integrating by parts and using 
$$
\mathrm{Hess}^{\dagger}\circ \mathrm{tr}^{\dagger}=(\mathrm{tr}\circ\mathrm{Hess})^{\dagger},
$$
where we consider $\mathrm{tr}(\bullet)$ as a smooth zeroth order linear differential operator, one gets that the distribution $\Delta u$ is in fact a Borel function which coincides with
$$
\Delta u(x)=-\mathrm{tr}_x(\mathrm{Hess}(u)|_x),
$$
so that
$$
|\Delta u|\leq \sqrt{m}|\mathrm{Hess}(u)|\>\text{  $\mu$-a.e. in $M$.}
$$
2. If $u\in  \mathsf{H}^{2,p}_0(M)$, and if $\{u_k\}\subset \C^{\infty}_{\Ic}(M)$ is a sequence such that $\left\| u-u_k\right\|_{2,p}\to 0$, then obviously $\{u_k\}$ is Cauchy in $\mathsf{L}^p(M)$, $\{\mathrm{Hess}(u_k)\}$ is Cauchy in $\Gamma_{\mathsf{L}^p}(M, \mathrm{T}^{0,2}M)$, and using 
\begin{align}
|\Delta \psi|\leq \sqrt{m}|\mathrm{Hess}(\psi)|,\>\text{  for all $\psi\in\C^{\infty}(M)$}, \label{sws}
\end{align}
it also follows that $\{\Delta u_k\}$ is Cauchy in $\mathsf{L}^p(M)$. In particular, one necessarily has
$$
\left\| u-u_k\right\|_{p}\to 0,\> \left\| \mathrm{Hess}(u)-\mathrm{Hess}(u_k)\right\|_{p}\to 0, \>\left\| \Delta u-\Delta u_k\right\|_{p}\to 0. 
$$
\end{Remark}

Remark \ref{sobo}.2 immediately implies that Calder\'on-Zygmund inequalities always extend to $\mathsf{H}^{2,p}_0(M)$ in the following sense:

\begin{Corollary}\label{sob2} Let $1<p<\infty$. If one has (\ref{cz}), then this inequality extends from $\C^{\infty}_{\Ic}(M)$ to $\mathsf{H}^{2,p}_0(M)$ with the same constants.
\end{Corollary}

The following definition will be convenient (cf. \cite{Gu}):

\begin{Definition} \emph{a)} $M$ is said to admit a \emph{sequence $(\chi_n)\subset \mathsf{C}^{\infty}_{\mathrm{c}}(M)$ of Laplacian cut-off functions}, if $(\chi_n)$ has the following properties\emph{:}
\begin{itemize}
\item[\emph{(C1)}] $0 \le \chi_n(x) \le 1$ for all $n\in\IN$, $x \in M$,

\item[\emph{(C2)}] for all compact $K\subset M$, there is an $n_0(K)\in\IN$ such that for all $n\geq n_0(K)$ one has $\chi_n\mid_{K}= 1$,

\item[\emph{(C3)}] $\sup_{x \in M} \left|\Id\chi_n(x)\right|_x \to 0$ as $n\to \infty$,

 \item[\emph{(C4)}] $\sup_{x \in M} \left|\Delta \chi_n(x)\right| \to 0$ as $n\to\infty$. 
\end{itemize}

\emph{b)} $M$ is said to admit a \emph{sequence $(\chi_n)\subset \mathsf{C}^{\infty}_{\mathrm{c}}(M)$ of Hessian cut-off functions}, if $(\chi_n)$ has the above properties (C1), (C2), (C3), and in addition
\begin{itemize}
 \item[\emph{(C4\rq{})}] $\sup_{x \in M} \left|\mathrm{Hess}( \chi_n)(x)\right|_x \to 0$ as $n\to\infty$.  
\end{itemize}
\end{Definition}

\begin{Remark} By (\ref{sws}), any sequence of Hessian cut-off functions is automatically a sequence of Laplacian cut-off functions. 
\end{Remark}

One has:

\begin{Proposition}\label{laplacian} \emph{a)} If $M$ is geodesically complete with $\mathrm{Ric}\geq 0$, then $M$ admits a sequence of Laplacian cut-off functions.\\
\emph{b)} If $M$ is geodesically complete with $\left\|\mathrm{R}\right\|_{\infty}<\infty$, then $M$ admits a sequence of Hessian cut-off functions.
\end{Proposition}

\begin{proof} a) This result is included in \cite{Gu}. It relies on a rigidity result by Cheeger and Colding \cite{CC}.\\
b) Let $m=\dim M$. By a result of L.-F. Tam\footnote{
a completely different construction of an exhaustion function with bounded gradient and Hessian is also contained in \cite{CG-chopping}.}, see Proposition 26.49 in \cite{chow}, one has that there is a constant $C=C(\left\|\mathrm{R}\right\|_{\infty},m)>0$, such that for any $x_0\in M$ there is a smooth function $\tilde{d}=\tilde{d}_{x_0}:M\to [0,\infty)$ satisfying
\begin{align}
\Id(\bullet,x_0)+1\leq \tilde{d}\leq \Id(\bullet,x_0)+C,\> |\grad(\tilde{d})|\leq C, \>|\mathrm{Hess}(\tilde{d})|\leq C.
\end{align}
Pick now a smooth function $t:\IR\to\ [0,1]$ which is compactly supported, equal to $1$ in $[0,C+\frac{1}{2}]$, and zero on $[C+1,\infty)$. Then $\chi_n(x):=t(\tilde{d}(x)/n)$ has the required properties. 
\end{proof}

\begin{Proposition}\label{saa}\emph{a)} Assume that (\ref{cz}) holds for some $1<p<\infty$ and that $M$ admits a sequence of Laplacian cut-off functions. Then one has $\mathsf{H}^{2,p}_0(M)=\mathsf{H}^{2,p}(M)$, in particular (by Corollary \ref{sob2}), (\ref{cz}) extends to $\mathsf{H}^{2,p}(M)$ with the same constants.\\
\emph{b)} If $M$ admits a sequence of Hessian cut-off functions, then one has $\mathsf{H}^{2,p}_0(M)=\mathsf{H}^{2,p}(M)$ for all $1<p<\infty$. 
\end{Proposition}

\begin{proof} Part a) has been observed in \cite{Gu}. Part b) follows from the same argument: Given a smooth $f\in \mathsf{H}^{2,p}(M)$, pick a sequence $(\chi_n)$ of Hessian cut-off functions. With $f_n:=\chi_nf$, using 
$$
\mathrm{Hess}(f_n) =f\mathrm{Hess}(\chi_n) + \Id \chi_n\otimes \Id f+\Id f\otimes \Id \chi_n+ \chi_n\mathrm{Hess}(f),
$$
one easily gets that $f_n$ converges to $f$ in $\mathsf{H}^{2,p}(M)$.
\end{proof}

We immediately get:

\begin{Corollary}\label{dssss} If $M$ is geodesically complete with $\left\|\mathrm{R}\right\|_{\infty}<\infty$, then one has $\mathsf{H}^{2,p}_0(M)=\mathsf{H}^{2,p}(M)$ for all $1<p<\infty$.
\end{Corollary}

Let us continue with a connection between Calder\'on-Zygmund inequalities and global control of solutions to the Poisson equation: Classically, one uses \emph{local} $\mathsf{L}^{p}$-Calder\'{o}n-Zygmund type inequalities in order to get
higher \emph{local} regularity of solutions of the Poisson equation $\Delta u=f$. Indeed,
if $u$ is in $\mathsf{W}_{\mathrm{loc}}^{1,p}$ and $f$ is in $\mathsf{L}_{\mathrm{loc}}^{p}$ then a local Calder\'{o}n-Zygmund type inequality shows $\partial_i\partial_j u$ is in $\mathsf{L}_{\mathrm{loc}}^{p}$, proving that $u$ is in $\mathsf{W}_{\mathrm{loc}}^{2,p}$. In fact, one can use this way of concluding to derive global estimates to solutions of the Poisson equation:

\begin{Proposition}\label{saa3} Let $1<p<\infty$, and assume that $M$ satisfies $\mathrm{CZ}(p)$ and admits a sequence of Hessian cut-off functions. Let $u\in \C^{2}\left(  M\right)$ be a
solution of the Poisson equation
\[
\Delta u=f.
\]
If \thinspace$u,\left\vert \grad( u)\right\vert ,f\in \mathsf{L}^{p}\left(  M\right)  $
then $| \mathrm{Hess}(u) | \in \mathsf{L}^{p}\left(  M\right)$.
\end{Proposition}

\begin{Remark}
It is not completely clear to what extent the $\IL^{p}$ assumption on the gradient is technical and related to the method of proof. In any case, it would be interesting to find situations where it is automatically satisfied. Compare also with Corollary \ref{cor_gradestimate}.
\end{Remark}

\begin{proof}[Proof of Proposition \ref{saa3}]
Pick a sequence of Hessian cut-off functions $(\varphi_{k})$ and defining the corresponding sequence of compactly supported $\mathsf{C}^2$-functions $u_{k}=u\varphi_{k}$. Then applying to $u_{k}$ Calder\'{o}n-Zygmund inequality we obtain (using Corollary \ref{sob2})
\begin{align*}
&\left\Vert \varphi_{k}\mathrm{Hess}\left(  u\right)  \right\Vert _{p}\\
&\leq C\left(  \left\Vert u\mathrm{Hess}\left(  \varphi_{k}\right)  \right\Vert
_{p}+\left\Vert \left\vert \grad( u)\right\vert \left\vert \grad(
\varphi_{k})\right\vert \right\Vert _{p}+\left\Vert u\Delta\varphi
_{k}\right\Vert _{p}+\left\Vert \varphi_{k}\Delta u\right\Vert _{p}\right)  .
\end{align*}
Whence, taking the limit as $k\rightarrow \infty$, the claim follows from dominated convergence.
\end{proof}

In order to prove our next application of $\mathrm{CZ}(p)$, a gradient estimate, we will need the following $\IL^p$-interpolation result, which should be of an independent interest, and which will also be used later on to prove local $\mathrm{CZ}(p)$ inequalities:

\begin{Proposition}\label{lemma_interpolation}
\emph{a)} For any $2\leq p<\infty$ there is a constant $C=C(p)>0$ such that for any $\varepsilon>0$, $u\in\mathsf{C}^{\infty}_{\mathrm{c}}(M)$ one has
\begin{align}
\left\Vert \grad(u)\right\Vert_{p}\leq    \frac{C}{\varepsilon}\left\Vert u\right\Vert_{p}+C\varepsilon\left\Vert
\mathrm{Hess}\left(  u\right)  \right\Vert_{p} . \label{cz-3}
\end{align}
\emph{b)} Assume that either $M$ is geodesically complete or that $M$ is a relatively compact open subset of an arbitrary smooth Riemannian manifold. Then for any $1<p\leq  2$ there is a constant $C=C(p)>0$ such that for any $\varepsilon>0$, $u\in\mathsf{C}^{\infty}_{\mathrm{c}}(M)$ one has
\begin{align}
\left\Vert \grad(u)\right\Vert_{p}\leq    \frac{C}{\varepsilon}\left\Vert u\right\Vert_{p}+C\varepsilon\left\Vert
\Delta  u  \right\Vert_{p} . \label{cz-5}
\end{align}
\end{Proposition}

\begin{proof} a) Let $u\in \C_{\Ic}^{\infty}\left(  M\right)$ and, having fixed $\alpha>0$, consider
the smooth, compactly supported vector field
\[
X:=u\cdot\left(  \left\vert \grad(u)\right\vert ^{2}+\alpha\right)  ^{\frac{p-2}{2}
}\grad(u).
\]
Using the divergence theorem and elaborating, we obtain
\begin{align*}
&\int\left(  \left\vert \grad(u)\right\vert ^{2}+\alpha\right)  ^{\frac{p-2}
{2}}\left\vert \grad(u)\right\vert ^{2} \z \\
&  \leq  | p-2 |
\int\left\vert u\right\vert \left(  \left\vert \grad(u)\right\vert ^{2}
+\alpha\right)  ^{\frac{p-4}{2}}\left\vert \grad(u)\right\vert ^{2}|
\mathrm{Hess}\left(  u\right) |\z \\
&  +\int\left\vert u\right\vert \left\vert \Delta u\right\vert \left(
\left\vert \grad(u)\right\vert ^{2}+\alpha\right)  ^{\frac{p-2}{2}}\z.
\end{align*}
Letting $\alpha\rightarrow0$ and applying the monotone and dominated
convergence theorems we get\footnote{obviously, by monotone convergence, the same integral inequality holds if $p <2$. However, in this case, the right hand side could be infinite. For instance, in
$\mathbb{R}^2$, this happens if $p=1$ as one can see by taking  $u(x,y)= (x^2 +1)\varphi(x,y)$, where $0\leq \varphi \leq 1$ is a cut-off function satisfying $\varphi =1$ on $[0,1] \times [0,1]$.}
\begin{align}\label{cz-4}
&\int\left\vert \grad(u)\right\vert ^{p}\z\\
&\leq | p-2 |  \int 
\left\vert u\right\vert | \mathrm{Hess}\left(  u\right)  |
 \left\vert \grad(u)\right\vert ^{p-2}\z+\int  \left\vert
u\right\vert \left\vert \Delta u\right\vert   \left\vert \grad(
u)\right\vert ^{p-2}\z. \nn
\end{align}
Now, in both the integrands appearing in the right hand side of (\ref{cz-4}), we use the
Young inequality
\[
ab\leq\frac{1}{\varepsilon^{p^{\prime}}p^{\prime}}a^{p^{\prime}}
+\frac{\varepsilon^{q^{\prime}}}{q^{\prime}}b^{q^{\prime}}, \>a,b\geq 0
\]
with
\[
p^{\prime}=p/2,\text{ }q^{\prime}=p/\left(  p-2\right).
\]
We obtain that the right-hand side of (\ref{cz-4}) is 
\begin{equation}
\leq\frac{1}{\varepsilon^{p^{\prime}}p^{\prime}
}\int\left\vert u\right\vert ^{\frac{p}{2}}| \mathrm{Hess}\left(
u\right)  | ^{\frac{p}{2}}\z+\frac{1}{\varepsilon^{p^{\prime}
}p^{\prime}}\int\left\vert u\right\vert ^{\frac{p}{2}}\left\vert \Delta
u\right\vert ^{\frac{p}{2}}\z+2\frac{\varepsilon^{q^{\prime}}}{q^{\prime}}
\int\left\vert \grad(u)\right\vert ^{p}. \label{cz-5}
\end{equation}
Choose $0<\varepsilon=\varepsilon\left(  p\right)  <1$ so small that
\[
A_{p}:=1-2\left(  p-2\right)  \frac{\varepsilon^{q^{\prime}}}{q^{\prime}}>0,
\]
and let $B_{p}:=1/(\varepsilon^{p^{\prime}}p^{\prime})$. Then, inserting (\ref{cz-5}) into (\ref{cz-4}) we deduce
\begin{align*}
&A_{p}\int\left\vert \grad(u)\right\vert ^{p}\z\\
&\leq\left(  p-2\right)  B_{p}
\int\left\vert u\right\vert ^{\frac{p}{2}}| \mathrm{Hess}\left(
u\right) | ^{\frac{p}{2}}\z+B_{p}\int\left\vert u\right\vert
^{\frac{p}{2}}\left\vert \Delta u\right\vert ^{\frac{p}{2}}\z. \label{cz-6}
\end{align*}
Whence, using twice the Young inequality
\begin{equation}\label{csch}
ab\leq\frac{1}{2\varepsilon^{2}}a^{2}+\frac{\varepsilon^{2}}{2}b^{2}
\end{equation}
with any arbitrary $\varepsilon>0$ we conclude
\begin{align*}
&\int\left\vert \grad(u)\right\vert ^{p}\z\\
&\leq\varepsilon^{-2}C_{p}\int\left\vert
u\right\vert ^{p}\z+\varepsilon^{2}D_{p}\left(\int| \mathrm{Hess}\left(
u\right) | ^{p}\z+\int\left\vert \Delta u\right\vert ^{p}\z\right)
\end{align*}
where we have set
\[
C_{p}:=\frac{\left(  p-1\right)  B_{p}}{2A_{p}}\text{, }D_{p}:=\frac{\max\left(
p-2,1\right)  B_{p}}{2A_{p}}.
\]

b) Assume first that $M$ is geodesically complete. Then by Theorem 4.1 of \cite{CD} we have the multiplicative inequality
\begin{align}
\left\Vert \grad(u)\right\Vert_{p}\leq  C(p) \left\Vert u\right\Vert_{p}^{\frac{1}{2}}\left\Vert
\Delta u   \right\Vert_{p}^{\frac{1}{2}}, \label{multiplicative}
\end{align}
which completes the proof in this case, using once more inequality (\ref{csch}).\\
Assume now that $M$ is a relatively compact open subset of an arbitrary smooth and geodesically incomplete Riemannian manifold $(\overline M, g)$. Then as above it is sufficient to prove that (\ref{multiplicative}) remains valid on $M$, which can be seen, for instance, from the usual construction of complete metrics in a given conformal class: Assume that $\overline M$ is an incomplete open manifold, otherwise the conclusion is trivial by restriction. We consider a smooth, relatively compact exhaustion $M_{k} \subset\subset M_{k+1} \nearrow \overline M$  such that $M \subset\subset M_1$ and we pick any smooth function $\lambda:\overline M \to \mathbb{R}_{\geq 0}$ with the following properties: (i) $\lambda(x) = 0$ on $M_1$; (ii) for every $k \geq 1$, $\lambda(x) = c_k$ on $M_{2k+1} \setminus M_{2k}$, where a sequence $\{c_k\} \subset (0,\infty)$ with $c_k\nearrow \infty$ will be specified later. Next, we define a new metric on $\overline M$ by $g_{\lambda}:= \mathrm{e}^{2\lambda}  g$. By construction, $ g_{\lambda} =  g$ on $M$. Let $r_k = \mathrm{dist}_{g}(\partial M_{2k+1}, \partial M_{2k})$ and choose $\{c_k\}$ in such a way that
$\sum r_k \cdot  \mathrm{e}^{c_k} = +\infty$. Then, $g_{\lambda}$ is geodesically complete. Indeed, if $\gamma : [0,\infty) \to \overline M$ is a divergent path, $\gamma$ is forced to pass across every annulus $M_{2k+1}\setminus M_{2k}$. Therefore, its length satisfies $\ell(\gamma) \geq \sum r_k \cdot  \mathrm{e}^{c_k} =  \infty$ and this characterizes the geodesic completeness.
Since (\ref{multiplicative}) holds on $(\overline{M},g_{\lambda})$, by restriction it holds on its open subset set $M$, as claimed.\\
Another way to prove the validity of (\ref{multiplicative}) on relatively compact open subset of possibly incomplete smooth Riemannian manifold is to use the (less elementary) double construction as explained in the proof of Theorem \ref{relco} a). 
\end{proof}

We immediately get the following Corollary, a gradient estimate which is a variant of (\ref{cz-5}) for arbitrary $p$, and which is useful in establishing compactness results for solutions of the Dirichlet problem for the Poisson equation (cf. the proof of Theorem \ref{relco} a) below):

\begin{Corollary}\label{cor_gradestimate}
In the situation of Proposition \ref{lemma_interpolation}, assume that $\mathrm{CZ}(p)$ holds on $M$ for some $1<p<\infty$. Then, there exists a constant $C>0$, which only depends on the $\mathrm{CZ}(p)$-constants and $p$, such that the following inequality
\[
\|\grad(u)\|_p \leq C (\|\Delta u\|_p + \|u\|_p)
\]
holds for every $u\in \C^{\infty}_\mathrm{c}(M)$.
\end{Corollary}

The rest of this section is devoted to local aspects of $\mathrm{CZ}(p)$ inequalities. We start with the following result, where it is claimed that $\mathrm{CZ}(p)$ with $1< p<\infty$ always holds on relatively compact open subsets (where of course the corresponding constants cannot be controlled explicitely), and moreover that $\mathrm{CZ}(p)$ is stable under compact perturbations:

\begin{Theorem}\label{relco} Let $1< p<\infty $.\\
\emph{a)} $\mathrm{CZ}(p)$ holds on any relatively compact open subset $\Omega\subset M$. 
Moreover, if $\Omega\subset M$ is a relatively compact domain with smooth boundary $\partial\Omega$, then $\mathrm{CZ}(p)$ holds in the stronger form
\[
\left\Vert \mathrm{Hess}\left(  u\right)  \right\Vert _{p}\leq
C \left\Vert \Delta u\right\Vert _{p},
\]
for every $u \in \C^{\infty}_{\Ic}(\Omega)$ and for a constant $C>0$. In both cases, the constants depend quite implicitly on the geometry of $\Omega$ (cf. the proof).
\\
\emph{b)} Assume that either $p \geq 2$ or that $1<p<2$ and $M$ is geodesically complete. If there is a relatively compact open subset $\Omega\subset M$ such that $\mathrm{CZ}(p)$ holds on $M\setminus \overline{\Omega}$, then $\mathrm{CZ}(p)$ also holds on $M$. Here, a possible choice of Calder\'on-Zygmund constants on $M$ depends on those of $M\setminus \overline{\Omega}$, those of a relatively compact open neighborhood of $\Omega$, and the choice of a gluing function (cf. the proof for the details).
\end{Theorem} 

In the next Corollary we essentially rephrase part b) of Theorem \ref{relco} in more geometric terms. This formulation involves two geometric objects: (1) the connected sum of Riemannian manifolds, whose construction will be recalled in Appendix \ref{gluing}, and (2) the notion of an \emph{end} $E$ of a complete Riemannian manifold $M$ with respect to a compact domain $\Omega$: $E$ is any of the unbounded connected components of $M\setminus  \overline\Omega$. 

\begin{Corollary}
A complete Riemannian manifold supports $\mathrm{CZ}(p)$, $1<p<\infty$, if and only if each of its ends $E_1,...,E_k$, with respect to any smooth, compact domain $\Omega$, supports the same Calder\'on-Zygmund  inequality. In particular, $\mathrm{CZ}(p)$ holds on the Riemannian connected sum $M = M_1 \# M_2$ of $m$-dimensional complete Riemannian manifolds $M_1$ and $M_2$ if and only if the same inequality holds on both $M_1$ and $M_2$. 
\end{Corollary}

The proof of part a) of Theorem \ref{relco} relies on the validity of the corresponding $\mathrm{CZ}(p)$ on a closed Riemannian manifold. This reduction procedure is obtained by using the Riemannian double of a manifold with boundary; see Appendix \ref{gluing}. See also Remark \ref{rem_otherproof} for a different and somewhat more direct argument. On the other hand, for the proof of part b), we will again need the $\mathsf{L}^p$-interpolation inequality from Proposition \ref{lemma_interpolation} a), which make gluing methods accessible to Calder\'on-Zygmund inequalities at all.

\begin{proof}[Proof of Theorem \ref{relco}]
a) Let $N$ be a relatively compact domain of $M$ such that $\partial N$ is a
smooth hypersurface and $\overline{\Omega}\subset N$.
\textquotedblleft The\textquotedblright\ Riemannian double $\mathcal{D}\left(  N\right)  $ of
$N$ is a compact Riemannian manifold without boundary. Moreover, by its very
construction, it is always possible to assume that $\mathcal{D}\left(
N\right)  $ contains an isometric copy $\Omega_{N}$ of the original domain
$\Omega$; see Appendix \ref{gluing}. We shall see in Theorem \ref{centrale1} below that
every closed manifold supports $\mathrm{CZ}\left(  p\right)  $. In particular,
this applies to $\mathcal{D}\left(  N\right)  $, namely, there exist suitable
constants $C_{1}>0$ and $C_{2}\geq0$, depending on the geometry of
$\mathcal{D}\left(  N\right)  $, such that
\begin{equation}
\left\Vert \mathrm{Hess}\left(  u\right)  \right\Vert _{p}\leq C_{1}\left\Vert
\Delta u\right\Vert _{p}+C_{2}\left\Vert u\right\Vert _{p}, \label{cz-domain}
\end{equation}
for every $u\in\C^{\infty}\left(  \mathcal{D}\left(  N\right)  \right)  $. In
particular, the same inequality holds for every $u\in\C_{\Ic}^{\infty}\left(
\Omega_{N}\right)  $. Since, up to Riemannian isometries, $\Omega_{N}$ is just
the original domain $\Omega$ , we conclude that (\ref{cz-domain}) (with the
same constants) holds on $\Omega$, as required.\\
We now assume that $\Omega$ as above has smooth boundary and is connected. Then, in spirit of the proof of Lemma 9.17 in
\cite{GT} we obtain that there exists a constant $C_{3}>0$ such that
\begin{equation}
\left\Vert u\right\Vert _{p}\leq C_{3}\left\Vert \Delta u\right\Vert
_{p},\label{cz-domain2}
\end{equation}
for every $u\in\C_{\Ic}^{\infty}\left(  \Omega\right)  $. Inserting this
latter into (\ref{cz-domain}) concludes the proof of part a). For the sake of completeness, let us provide a self-contained proof of (\ref{cz-domain2}). By contradiction, suppose that there exists a sequence $\left\{
u_{k}\right\}  \subset\C_{\Ic}^{\infty}\left(  \Omega\right)  $ satisfying
\begin{equation}
\text{(a) }\left\Vert u_{k}\right\Vert _{p}=1\text{,\qquad(b) }\left\Vert
\Delta u_{k}\right\Vert _{p}\rightarrow0.\label{p-norm0}
\end{equation}
Note that, by Corollary \ref{cor_gradestimate}, $\{  u_{k}\}  $ is
bounded in $\mathsf{\mathsf{H}}_{0}^{1,p}\left(  \Omega\right)  $. Therefore,
the Rellich-Kondrachov compactness theorem yields the existence of a
subsequence $\{  u_{k^{\prime}}\}  $ that converges strongly in $\IL^p$ to a function
$u\in \IL^p(\Omega)$.
In fact, $u\in\C^{0}(  \overline{\Omega})  $ if $p>m$. It follows
from (\ref{p-norm0}) (a) that
\begin{equation}
\left\Vert u\right\Vert _{p}=1.\label{p-norm1}
\end{equation}
Now, by (\ref{cz-domain}) and Corollary \ref{cor_gradestimate}, $\{  u_{k^{\prime}}\}  $ is bounded in the reflexive Banach
space $\mathsf{H}_{0}^{2,p}\left(  \Omega\right)  $. Therefore, a subsequence $\{
u_{k^{\prime\prime}}\}  $ converges weakly in $\mathsf{H}_{0}^{2,p}\left(  \Omega\right) $ and the weak limit is $u\in \mathsf{H}_{0}^{2,p}(\Omega)$. In particular, by Remark \ref{sobo}.1, we have that the distributional Laplacian of $u$ is a Borel function $\Delta u \in \IL^p(\Omega)$ and, furthermore, for
every $\varphi\in\C_{\Ic}^{\infty} ( \Omega)  $,
\begin{align*}
\int_{\Omega}\varphi \Delta u_{k^{\prime\prime}} \Id\mu &= -\int_{\Omega}\mathrm{tr \circ Hess}(u_{k^{\prime\prime}}) \cdot\varphi \ \Id  \mu \\
&= \int_{\Omega}\left(-\mathrm{Hess}(u_{k^{\prime\prime}}),\mathrm{tr}^{\dagger}(\varphi)\right)\Id \mu\\
&\to\int_{\Omega} \left(-\mathrm{Hess}(u),\mathrm{tr}^{\dagger}(\varphi)\right)\Id \mu\\
&=-\int_{\Omega}\mathrm{tr \circ Hess}(u) \cdot\varphi \ \Id \mu\\
&= \int_{\Omega} \varphi \Delta u \ \Id \mu.
\end{align*}
On the other hand, by (\ref{p-norm0}) (b),
\[
\int_{\Omega}\varphi \Delta u_{k^{\prime\prime}} \Id\mu \to 0,
\]
thus proving that $u\in \mathsf{H}_{0}^{2,p}\left(  \Omega\right)  $ is a strong solution of the Laplace equation:
\[
\Delta u=0 \text{ a.e. in } \Omega.
\] 
By elliptic regularity, Theorem 9.19 in \cite{GT}, since the Laplace-Beltrami operator is uniformly elliptic with smooth coefficients in $\overline{\Omega}$, we deduce that $u \in \C^{\infty}(\overline{\Omega})$ and $u=0$ on $\partial \Omega$. The usual maximum principle then implies that $u=0$. Obviously, this contradicts (\ref{p-norm1}).\\
b) Let $u\in\C^{\infty}_{\Ic}(M)$. Take an open subset $\Omega_1\subset M$ such that $\Omega\subset \subset \Omega_1\subset \subset M$ and a function $\xi\in\mathsf{C}^{\infty}_{\Ic}(\Omega_1)$ with $0\leq \xi\leq 1$, and $\xi=1$ on $\Omega$, which we are going to use in oder to glue two $\mathrm{CZ}(p)$\rq{}s together. To this end, set $\phi:=(1-\xi)$. By the validity of $\mathrm{CZ}(p)$ on $\Omega_1$ (by part a)) and on $M\setminus \overline{\Omega}$, and writing $u=\xi u+ \phi u$, with $\xi u\in\C^{\infty}_{\Ic}(\Omega_1)$, $\phi  u\in\C^{\infty}_{\Ic}(M\setminus \overline{\Omega})$, we get $A,B>0$, independent of $u$, such that
\begin{align*}
&\left\|\mathrm{Hess}(u)\right\|_p\leq \left\|\mathrm{Hess}(\xi u)\right\|_p + \left\|\mathrm{Hess}(\phi u)\right\|_p\\
&\leq A\left(\left\|\Delta(\xi u)\right\|_p+\left\|\xi u\right\|_p\right)+ B\left(\left\|\Delta(\phi  u)\right\|_p+\left\|\phi u\right\|_p\right)\\
&\leq A\left(\left\|u\Delta\xi \right\|_p+\left\|\xi\Delta u\right\|_p+\left\||\grad(\xi)|  \cdot|\grad(u)|\right\|_p+\left\|\xi u\right\|_p\right)\\
&\>\>\>+B\left(\left\|u\Delta\phi \right\|_p+\left\|\phi\Delta u\right\|_p+\left\||\grad(\phi)| \cdot|\grad(u)|\right\|_p+\left\|\phi u\right\|_p\right)\\
&\leq C \left(\left\|\Delta u\right\|_p+\left\|\grad(u)\right\|_p+\left\| u\right\|_p\right),
\end{align*}
with
\begin{align*}
0<C:=&A\max\big(\left\Vert  \Delta \xi\right\Vert
_{\infty} ,\left\Vert    \xi\right\Vert
_{\infty},\left\Vert  \grad(\xi)\right\Vert
_{\infty}\big)\\
&+B\max\big(\left\Vert  \Delta \phi\right\Vert
_{\infty} , \left\Vert   \phi\right\Vert
_{\infty}, \left\Vert  \grad(\phi)\right\Vert
_{\infty}\big).
\end{align*}
Interpolating with Proposition \ref{lemma_interpolation} a), the latter inequality completes the proof.
\end{proof}

We have seen that $\mathrm{CZ}(p)$ always holds on relatively compact domains, however, in general one may have a rough control on the constants. We close this section with the following Theorem \ref{local} where we prove a much more precise $\mathrm{CZ}(p)$ on sufficiently small geodesic balls. To this end recall the definition of $r_{Q,k,\alpha}(x)$, the $\mathsf{C}^{k,\alpha}$-harmonic radius with accuracy $Q$ at $x$ (cf. Appendix \ref{harm}).

\begin{Theorem}\label{local} Fix an arbitrary $x\in M$. Then for all $1< p<\infty$, all 
\begin{align}
0<r<r_{2,1,1/2}(x)/2,\label{rad}
\end{align}
and all real numbers $D$ with
$$
r_{2,1,1/2}\big(\mathrm{B}_{r_{2,1,1/2}}(x)\big)=\inf_{z\in \mathrm{B}_{r_{2,1,1/2}(x)}(x) }r_{2,1,1/2}(z)\geq D>0, 
$$
there is a constant $C=C(r,p,m,D)>0$, such that for all $u\in\C^{\infty}_{\Ic}(M)$ one has
\begin{align}\nn
&\left\Vert 1_{\mathrm{B}_{r/2}(x)}\mathrm{Hess}\left(  u\right)  \right\Vert _{p}\\
&\leq  C\left(\left\Vert 1_{\mathrm{B}_{2r}(x)} u\right\Vert _{p}+\left\Vert 1_{\mathrm{B}_{2r}(x)} \Delta u\right\Vert
_{p}+\left\Vert 1_{\mathrm{B}_{2r}(x)} \grad( u)\right\Vert
_{p}\right).\label{loci}
\end{align}
In particular, with some $\tilde{C}=\tilde{C}(r,p,m,D)>0$, for all $u\in\C^{\infty}_{\Ic}(\mathrm{B}_{r/2}(x))$ one has
\begin{align}
\left\Vert \mathrm{Hess}\left(  u\right)  \right\Vert _{p}\leq
 C(\left\Vert u\right\Vert _{p}+\left\Vert \Delta u\right\Vert
_{p}).\label{loci2}
\end{align}
\end{Theorem}

\begin{proof} Let $u\in\C^{\infty}_{\Ic}(M)$, let $r^*(x):=r_{2,1,1/2}(x)$, and pick a $\mathsf{C}^{1,1/2}$-harmonic coordinate system 
$$
\phi=(y^1,\dots,y^m):\mathrm{B}_{r^*(x)}(x)\longrightarrow  \IR^m
$$
with accuracy $Q=2$. Then by the properties (\ref{hari1}) and (\ref{hari2}) of $\phi$ and by Remark \ref{rem_inverse} we have the following inequalities on $\mathrm{B}_{r^*(x)}(x)$,
\begin{align}\label{ede1}
&2^{-1}(\delta_{ij})\leq (g_{ij})\leq 2 (\delta_{ij}),\\
\label{ede2}&\max_{i,j\in\{1,\dots,m\}}\{\left\Vert g_{ij}\right\Vert_{\infty},\left\Vert g^{ij}\right\Vert_{\infty}\}\leq C_1(m),\\
\label{ede3}&\max_{i,j,l\in\{1,\dots,m\}}\{ \left\Vert \partial_lg_{ij}\right\Vert _{\infty},\left\Vert \partial_lg^{ij}\right\Vert _{\infty}\} \leq C_2(D,m).
\end{align}
Since by (\ref{ede1}) and $(\ref{rad})$ we have 
$$
\phi(\mathrm{B}_{r/2}(x))\subset\mathrm{B}^{\mathrm{eucl}}_{r/\sqrt{2}}(0)\subset  \mathrm{B}^{\mathrm{eucl}}_{\sqrt{2}r}(0)\subset \phi(\mathrm{B}_{2r}(x))\subset \phi(\mathrm{B}_{r^*(x)}(x)),
$$
applying Theorem 9.11 from \cite{GT} with $L=\Delta$ and with the Euclidean balls $\Omega:=\mathrm{B}^{\mathrm{eucl}}_{\sqrt{2}r}(0)$, $\Omega\rq{}:=\mathrm{B}^{\mathrm{eucl}}_{r/\sqrt{2}}(0)$) implies the existence of a $C_3=C_3(p,r,m,D)>0$ such that
\begin{align}\label{elli}
\int_{\mathrm{B}_{r/2}(x)}\sum_{i,j}|\partial_i \partial_j u(y)|^p \Id y \leq &\>\>C_3  \int_{\mathrm{B}_{2r}(x)}|\Delta u(y)|^p\Id y+C_3  \int_{\mathrm{B}_{2r(x)}} |u(y)|^p \Id y.
\end{align}
One can deduce from (\ref{ede1})-(\ref{ede3}) the following pointwise estimate in $\mathrm{B}_{r^*(x)}(x)$,
\begin{align}\label{pointwise-hess-estimate}
&|\mathrm{Hess}(u)|^p\leq C_8 
\sum_{i,j} |\partial_i\partial_j u|^p+C_5 |\mathrm{grad}(u)|^p,
\end{align}
with some $C_8=C_8(r,m,D,p)>0$.
Indeed,  let $\left\Vert A\right\Vert^2 _{\mathrm{HS}}=\sum_{ij} A_{ij}^2$ denote the Hilbert-Schmidt norm of a real-valued matrix $A=(A_{ij})$. Then with $H:=(\mathrm{Hess}\left(  u\right)_{ij})$,
$h:=(\partial_{i}\partial_{j}u)$, $\Gamma:=(-\sum_l\Gamma_{ij}^{l}\partial_{l}u)$,
$G:=(g_{ij})$ one has $H=h+\Gamma$ and 
\begin{align*}
&|\mathrm{Hess}(u)|=\left\Vert G^{-1}
H\right\Vert _{\mathrm{HS}}=\left\Vert G^{-1}h+G^{-1}\Gamma\right\Vert _{\mathrm{HS}}\\
&\leq\left\Vert G^{-1}\right\Vert _{\mathrm{HS}}\left(  \left\Vert h\right\Vert
_{\mathrm{HS}}+\left\Vert \Gamma\right\Vert _{\mathrm{HS}}\right) \leq C_4 \left( \left\Vert h\right\Vert _{\mathrm{HS}}+\left\Vert
\Gamma\right\Vert _{\mathrm{HS}}\right)  \>\text{ in $\mathrm{B}_{r^*(x)}(x)$}
\end{align*}
and
\[
\left\Vert \Gamma\right\Vert _{\mathrm{HS}}^{2}\leq\sum_{ij}\big(\sum_{l}\Gamma_{ij}
^{l}\big)^{2}\cdot\sum_{l}\left(  \partial_{l} u\right)  ^{2}\leq C_5  \left\vert \grad(
u)\right\vert^{2}\>\text{ in $\mathrm{B}_{r^*(x)}(x)$},
\]
for some $C_4,C_5>0$ depending only on $D$ and $m$. Whence, we get the estimate on $\mathrm{B}_{r^*(x)}(x)$
\begin{align*}
& |\mathrm{Hess}(u)|\leq C_6\sqrt{\sum_{i,j} |\partial_i\partial_j u|^2}+C_6|\mathrm{grad}(u)|,\\
&|\mathrm{Hess}(u)|^p\leq C_7 \sum_{i,j} |\partial_i\partial_j u|^p+C_7|\mathrm{grad}(u)|^p
\end{align*}
for some $C_6,C_7>0$ depending only on $D$, $m$ and $p$. This proves the validity of (\ref{pointwise-hess-estimate}).
Using this latter in combination with (\ref{ede1}) and (\ref{elli}) gives us a $C_9=C_9(r,m,D,p)>0$ such that 
\begin{align*}
&\left\Vert 1_{\mathrm{B}_{r/2}(x)} \mathrm{Hess}\left(  u\right)  \right\Vert _{p}\\
&\leq C_9\left(\left\Vert 1_{\mathrm{B}_{2r}(x)}  u\right\Vert _{p}+\left\Vert 1_{\mathrm{B}_{2r}(x)} \Delta u\right\Vert
_{p}+\left\Vert 1_{\mathrm{B}_{2r}(x)}  \mathrm{grad}(u)\right\Vert
_{p}\right).
\end{align*}
Finally, (\ref{loci2}) follows from by interpolation using Proposition \ref{lemma_interpolation} a).  
\end{proof}

Let us remark here that an essential point of the estimate from Theorem \ref{local} is that $C$ depends on $x$ only through a lower bound $D$ on the local harmonic radius, a fact which makes it possibly to use this result in order to derive $\mathrm{CZ}(p)$ on a large class of noncompact Riemannian manifolds (cf. the proof of Theorem \ref{centrale1}).

\section{Geometric criteria for global Calder\'on-Zygmund inequalities}\label{gross}

This section is devoted to Riemann geometric criteria for the validity of global $\mathrm{CZ}(p)$ inequalities.\\
Let us start with the $p=2$ case: Here, in view of Bochner\rq{}s equality, it is easy to give a rather complete answer: $\mathrm{CZ}(2)$ always holds globally in a strong, \lq\lq{}infinitesimial\rq\rq{} way, under a global lower bound on the Ricci curvature, and furthermore this result does not even require geodesic completeness:

\begin{Proposition} \label{Ricbelow} Assume that $\mathrm{Ric}\geq -C^2$ for some constant $C\in\IR$, meaning as usual that 
$$
\mathrm{Ric}(X,X)\geq -C^2 |X|^2\text{ for all vector fields $X\in\Gamma_{\C^{\infty}}(M,\T M)$.}
$$
Then $\mathrm{CZ}(2)$ holds in the following \lq\lq{}infinitesimal\rq\rq{} way: For every $\varepsilon>0$ and every $u\in\C^{\infty}_{\Ic}(M)$ one has
\[
\left\Vert \mathrm{Hess}\left(  u\right)  \right\Vert _{2}^{2}\leq
\frac{C\varepsilon^{2}}{2}\left\Vert u\right\Vert _{2}^{2}+\left(
1+\frac{C^{2}}{2\varepsilon^{2}}\right)  \left\Vert \Delta u\right\Vert
_{2}^{2}.
\]
\end{Proposition}

\begin{proof} By Bochner\rq{}s equality we have
\begin{align*}
&|\mathrm{Hess}(u)|^2\\
&=-\f{1}{2}\Delta |\mathrm{grad}(u)|^2+(\mathrm{grad}(u),\mathrm{grad}(\Delta u)) - \mathrm{Ric}(\mathrm{grad}(u),\mathrm{grad}(u))\\
&=-\f{1}{2}\Delta |\mathrm{grad}(u)|^2+(\Id u,\Id \Delta u) - \mathrm{Ric}(\mathrm{grad}(u),\mathrm{grad}(u)).
\end{align*}
Now the claim follows easily from integrating this identity, using integration by parts, $\Delta u=\Id^{\dagger}\Id u$ and the inequality
$$
ab \leq \f{a^2}{2\varepsilon^2}+\f{\varepsilon^2 b^2}{2},
$$
valid for $a,b\geq 0$.
\end{proof}

On the other hand, it is necessary for $\mathrm{CZ}(2)$ to have \emph{some} control on the curvature, as can be seen from:

\begin{Theorem}\label{counter} There exists a $2$-dimensional, geodesically complete Riemannian manifold $N$ with unbounded Gaussian curvature and such that $\mathrm{CZ}(2)$ fails on $N$.
\end{Theorem}

The proof of Theorem \ref{counter} is given in Section \ref{model}.\\

For arbitrary values of $p$, the situation is much more complicated. Here, we found the following two criteria, which can also be considered as the main result of this paper. \\
The first result covers the whole $\mathsf{L}^p$-scale in a great generality:

\begin{Theorem}\label{centrale1} Let $1<p<\infty$ and assume $\left\|\mathrm{Ric}\right\|_{\infty}<\infty$, $r_{\mathrm{inj}}(M)>0$. Then there is a 
$$
C=C(m,p,\left\|\mathrm{Ric}\right\|_{\infty},r_{\mathrm{inj}}(M))>0, 
$$
such that for all $u\in\C^{\infty}_{\Ic}(M)$ one has
\begin{align}
\left\Vert \mathrm{Hess}\left(  u\right)  \right\Vert _{p}\leq
 C(\left\Vert u\right\Vert _{p}+\left\Vert \Delta u\right\Vert
_{p}).\label{loww}
\end{align}
 \end{Theorem}
 
\begin{Remark} Note that, using the usual definition of geodesic completness in terms of the exponential function, it is elementary to see that a positive injectivity radius automatically implies geodesically completeness.
\end{Remark} 

The second result is concerned with the $1<p\leq 2$ case in a slightly different setting: the geometry of the manifold is bounded up to order one but the injectivity radius condition is replaced with a kind of generalized volume doubling assumption. The proof of this result is of independent interest because it points out a deep relation between Calder\'on-Zygmund inequalities and covariant Riesz transforms:

\begin{Theorem}\label{centrale2} Let $1<p\leq 2$. Assume that $M$ is geodesically complete with $\left\|\mathrm{R}\right\|_{\infty}<\infty$, $\left\|\nabla\mathrm{R}\right\|_{\infty}<\infty$, and that there are constants $D\geq 1$, $0\leq\delta<2$ with 
\begin{align}
\mu(\mathrm{B}_{tr}(x))\leq D t^{D}\mathrm{e}^{t^{\delta
}+r^{\delta}}\mu(\mathrm{B}_r(x))\>\text{ for all $x\in M$, $r>0$, $t\geq1$.}\label{B2}
\end{align}
Then there is a 
$$
C=C(m,p,\left\|\mathrm{R}\right\|_{\infty},\left\|\nabla \mathrm{R}\right\|_{\infty},D,\delta)>0,
$$
such that for all $u\in\C^{\infty}_{\Ic}(M)$ one has
\begin{align}
\left\Vert \mathrm{Hess}\left(  u\right)  \right\Vert _{p}\leq
 C(\left\Vert u\right\Vert _{p}+\left\Vert \Delta u\right\Vert
_{p} ).\label{loww2}
\end{align}
 \end{Theorem}

\begin{Remark} If $M$ is geodesically complete with $\mathrm{Ric}\geq 0$, then one has the doubling condition
$$
\mu(\mathrm{B}_{2r}(x))\leq 2^m\mu(\mathrm{B}_r(x))\text{ for all $r>0$, $x\in M$,}
$$
which easily implies
$$
\mu(\mathrm{B}_{tr}(x))\leq 2^{2m}t^m\mu(\mathrm{B}_r(x))\text{ for all $r>0$, $t\geq 1$, $x\in M$,}
$$
so that (\ref{B2}) is satisfied in this situation (with constants that only depend on $m$). 
\end{Remark}

On the other hand, nonnegative Ricci curvature is not necessary for (\ref{B2}):

\begin{Example}\label{negric}
Let $\left(  N,h\right)  $ be a compact Riemannian manifold of dimension $m-1$
and let $\left(  M,g\right)  =\left(  N\times \IR,h+\Id t\otimes \Id t\right)  $. Then,
$\left(  M,g\right)$ satisfies the assumptions of Theorem \ref{centrale2} (actually, also those of Theorem \ref{centrale1}). Indeed, $M$ is co-compact, hence it has bounded
geometry up to order $\infty$. On the other hand, there exists a constant $C>0$
depending on the geometry of $N$ such that, for every $\left(  p_{0}
,t_{0}\right)  \in M$, $R>0$ and $t\geq1$
\[
\frac{\mu_M\left(  \mathrm{B}_{tR}^{M}\left(  \left(  p_{0},t_{0}\right)
\right)  \right)  }{\mu_M\left(  \mathrm{B}_{R}^{M}\left(  \left(  p_{0}
,t_{0}\right)  \right)  \right)  }\leq Ct^{m}.
\]
Indeed, since
\[
\max( \Id_{N},\Id_{\mathbb{R}})  \leq \Id_{M}=\sqrt{\Id_{N}
^{2}+\Id_{\mathbb{R}}^{2}}\leq\sqrt{2}\max(  \Id_{N},\Id_{\mathbb{R}})
\]
we have
\[
 \mathrm{B}_{R/\sqrt{2}}^{N}\left(  p_{0}\right)  \times  \mathrm{B}_{R/\sqrt{2}}^{\mathbb{R}
}\left(  t_{0}\right)  \subseteq \mathrm{B}_{R}^{M}\left(  \left(  p_{0},t_{0}\right)
\right)  \subseteq  \mathrm{B}_{R}^{N}\left(  p_{0}\right)  \times  \mathrm{B}_{R}^{\mathbb{R}
}\left(  t_{0}\right)
\]
proving that
\[
\frac{\mu_M\left( \mathrm{B}_{tR}^{M}\left(  \left(  p_{0},t_{0}\right)  \right)\right)}{\mu_M\left( \mathrm{B}_{R}^{M}\left(  \left(  p_{0},t_{0}\right)  \right)  \right)}\leq
\sqrt{2}t\cdot\frac{\mu_N\left( \mathrm{B}_{tR}^{N}\left(  p_{0}\right) \right) }
{\mu_N\left( \mathrm{B}_{R/\sqrt{2}}^{N}\left(  p_{0}\right)  \right)}.
\]
Therefore, we are reduced to show that
\[
\frac{\mu_N\left( \mathrm{B}_{tR}^{N}\left(  p_{0}\right) \right) }{\mu_N\left( \mathrm{B}_{R/\sqrt{2}
}^{N}\left(  p_{0}\right)  \right)}\leq Ct^{m-1}.
\]
To this end, note that, if
\[
tR\leq\frac{r_{\mathrm{in}j}\left(  N\right)  }{2}.
\]
then, the desired inequality follows from volume comparison. Indeed, let
\[
-K^{2}\leq\mathrm{Sec}_{N}\leq K^{2},
\]
then the continuous functions $\alpha_{1},\alpha_{2}:[0,r_{\mathrm{inj}}\left(
N\right)  /2]\rightarrow (0,\infty)$ defined by (P. Petersen notation, \cite{Petersen})
\begin{align*}
\alpha_{1}\left(  r\right)   &  =\frac{\int_{0}^{r}\mathrm{sn}_{K^{2}}
^{m-2}\left(  s\right)  \Id s}{r^{m-1}}\\
\alpha_{2}\left(  r\right)   &  =\frac{\int_{0}^{r}\mathrm{sn}_{-K^{2}}
^{m-2}\left(  s\right) \Id s}{r^{m-1}}
\end{align*}
satisfy
\[
A_{i}\leq\alpha_{i}\left(  r\right)  \leq B_{i}
\]
where $A_{i},B_{i}>0$ are constants depending only on $K,$ $m$ and
$r_{\mathrm{inj}}\left(  N\right)  $. It follows that
\[
\frac{\mu_N\left( \mathrm{B}_{tR}^{N}\left(  p_{0}\right) \right) }{\mu_N\left(\mathrm{B}_{R/\sqrt{2}
}^{N}\left(  p_{0}\right)\right)  }\leq\frac{B_{2}}{A_{1}}\frac{\left(  tR\right)
^{m-1}}{\left(  R/\sqrt{2}\right)  ^{m-1}}=Ct^{m-1},
\]
as claimed. On the other hand, if

\[
tR>\frac{r_{\mathrm{inj}}\left(  N\right)}{2},
\]
since
\[
\frac{\mu_N\left( \mathrm{B}_{tR}^{N}\left(  p_{0}\right)  \right)}{\mu_N\left( \mathrm{B}_{R/\sqrt{2}
}^{N}\left(  p_{0}\right) \right) }\leq\frac{\mu_N(N)}{\mu_N\left(
 \mathrm{B}_{r_{\mathrm{inj}}\left(  N\right)  /(t2\sqrt{2})}^{N}\left(  p_{0}\right) \right) }
\]
and
\[
\frac{r_{\mathrm{inj}}\left(  N\right)  }{t2\sqrt{2}}\leq\frac{r_{\mathrm{inj}}\left(  N\right)
}{2}
\]
using again volume comparison we get
\[
\mu_N\left(\mathrm{B}_{r_{\mathrm{inj}}\left(  N\right)  /t2\sqrt{2}}^{N}\left(  p_{0}\right)\right)
\geq A_{1}\left(  \frac{r_{\mathrm{inj}}\left(  N\right)  }{t2\sqrt{2}}\right)  ^{m-1}
\]
and, hence,
\[
\frac{\mu_N\left( \mathrm{B}_{tR}^{N}\left(  p_{0}\right)  \right)}{\mu_N\left( \mathrm{B}_{R/\sqrt{2}}^{N}\left(  p_{0}\right)  \right)}\leq\frac{\mu_N(N)}{A_{1}}\left(
\frac{2\sqrt{2}}{r_{\mathrm{inj}}\left(  N\right)  }\right)  ^{m-1}\cdot t^{m-1}.
\]
This completes the proof.
\end{Example}

The rest of this section is devoted to the proof of Theorem \ref{centrale1} and of Theorem \ref{centrale2}, respectively. \\
We will need the following auxiliary result (see for example Lemma 1.6 in \cite{hebey} and its proof) for the former:

\begin{Lemma}\label{eses} Assume that $M$ is geodesically complete with $\mathrm{Ric}\geq -C$ for some $C>0$. Then for
any $r>0$ there exists a sequence of points $\{x_i\}\subset M$ and a natural number $N =N(m,r,C) <\infty$, such that 
\begin{itemize}
\item $\mathrm{B}_{r/4}\left(  x_{i}\right)  \cap \mathrm{B}_{r/4}\left(  x_{j}\right)
=\emptyset$  for all $i,j\in\IN$ with $i\neq j$,

\item $\bigcup_{i\in\IN} \mathrm{B}_{r/2 }\left(  x_{i}\right)  =M$,

\item the intersection multiplicity of the system $\{\mathrm{B}_{2r }(x_{i})|i\in\IN\}$ is $\leq N$.
\end{itemize}
\end{Lemma}

Now we can give the

\begin{proof}[Proof of Theorem \ref{centrale1}] By Theorem \ref{univ} there is a 
$$
D=D(m,r_{\mathrm{inj}}(M),\left\|\mathrm{Ric}\right\|_{\infty})>0
$$
such that $r_{2,1,1/2}(M)\geq D$. Let $r:=D/2$. We take a covering $\cup_{i\in\IN} \mathrm{B}_{r/2}\left(  x_{i}\right)  =M$ as in Lemma \ref{eses}. By Theoren \ref{local} we have a $c=c(r,p,m,D)>0$ such that, for all $i\in\IN$, all $u\in\C^{\infty}_{\Ic}(M)$,
\begin{align*}
&\int_{\mathrm{B}_{r/2}\left(  x_{i}\right)} \left|\mathrm{Hess}\left(  u\right)\right|^p\Id \mu\\
&\leq c \int_{\mathrm{B}_{2r}\left(  x_{i}\right)} \left|\Delta u\right|^p\Id \mu+c\int_{\mathrm{B}_{2r}\left(  x_{i}\right)} \left|\mathrm{grad}\left(  u\right)\right|^p\Id \mu +c\int_{\mathrm{B}_{2r}\left(  x_{i}\right)} \left|u\right|^p\Id \mu, 
\end{align*} 
so summing over $i$ and using monotone convergence we get  
\begin{align*}
&\int_{M} \left|\mathrm{Hess}\left(  u\right)\right|^p\Id \mu\leq \sum_i\int_{\mathrm{B}_{r/2}\left(  x_{i}\right)} \left|\mathrm{Hess}\left(  u\right)\right|^p\Id \mu\\
&\leq c \int_{M} \sum_i 1_{\mathrm{B}_{2r}\left(  x_{i}\right)} \left|\Delta u\right|^p\Id \mu+c\int_{M} \sum_i 1_{\mathrm{B}_{2r}\left(  x_{i}\right)} \left|\mathrm{grad}\left(  u\right)\right|^p\Id \mu\\
&\>\>\>+c\int_{M} \sum_i 1_{\mathrm{B}_{2r}\left(  x_{i}\right)} \left|u\right|^p\Id \mu, 
\end{align*} 
which by Lemma \ref{eses} gives
$$
\left\Vert \mathrm{Hess}\left(  u\right)  \right\Vert _{p}    \leq (cN)^{1/p} \left( \left\Vert  \Delta u\right\Vert
_{p}   +    \left\Vert    \grad( u) \right\Vert _{p}
 +   \left\Vert u\right\Vert _{p}\right).
$$
A use of Proposition \ref{lemma_interpolation} a) completes the proof.
\end{proof}

\begin{Remark} \label{rem_otherproof}
Obviously, a similar argument can be used to prove Theorem \ref{relco} a). Simply cover the compact domain $\overline \Omega$ with a finite number of balls $\mathrm{B}_{r/2}$ with $0<2r < r_{2,1,1/2}(\overline \Omega)$.
\end{Remark}

Finally, we give the proof of Theorem \ref{centrale2}, which as we have already remarked in the introduction, uses the machinery of covariant Riesz-transforms. We will need the following auxiliary Hilbert space lemma:

\begin{Lemma}\label{edh} Let $S$ be a densely defined closed linear operator from a Hilbert space $\IHH_1$ to a Hilbert space $\IHH_2$, and let $T$ be a bounded self-adjoint operator in $\IHH_1$. Then for any $\lambda>0$ with $T\geq  -\lambda$ one has
$$
\left\| S(S^* S +T+\lambda +1)^{-1/2} \right\| \leq 1.
$$
\end{Lemma}

\begin{proof} Firstly, the polar decomposition of $S$ reads $S=U(S^*S)^{1/2}$, with a partial isometry $U$ from $\IHH_1$ to $\IHH_2$ whose domain of isometry contains the range of $(S^* S)^{1/2}$. Secondly, we have 
$$
 S^* S +T+\lambda+1\geq   S^* S +1,
$$
which in this case means nothing but 
$$
\left\|(S^* S +T+\lambda+1)^{1/2} f \right\|\geq   \left\|(S^* S +1)^{1/2}f\right\|
$$
for all $f$ in the domain of definition of $(S^* S )^{1/2}$, in particular, 
$$
\left\|(S^* S +1)^{1/2}(S^* S +T+\lambda+1)^{-1/2} h\right\|\leq \left\| h \right\|\text{ for all $h\in  \IHH_2$} 
$$
so
\begin{align*}
\left\|(S^* S +1)^{1/2}(S^* S +T+\lambda+1)^{-1/2} \right\|\leq 1.
\end{align*}
Now we can estimate as follows
\begin{align*}
&\left\| S(S^* S +T+\lambda +1)^{-\f{1}{2}} \right\|\\
&=\left\| S(S^*S+1)^{-\f{1}{2}}(S^*S+1)^{\f{1}{2}}(S^* S +T+\lambda +1)^{-\f{1}{2} }  \right\|\\
&\leq \left\| (S^*S)^{\f{1}{2}}(S^*S+1)^{-\f{1}{2}}(S^*S+1)^{\f{1}{2}}(S^* S +T+\lambda +1)^{-\f{1}{2} }  \right\|\\
&\leq \left\| (S^*S)^{\f{1}{2}}(S^*S+1)^{-\f{1}{2}} \right\|\leq  \sup_{t\geq 0}  \sqrt{t} /  \sqrt{t+1} \leq 1,
\end{align*}
where we have used the spectral calculus for the last norm bound.
\end{proof}

\begin{proof}[Proof of Theorem \ref{centrale2}] The assumption $\left\|\mathrm{R}\right\|_{\infty}<\infty$ implies
\begin{align}\label{gal}
-c\leq \mathrm{Ric}\leq \tilde{c}\text{ for some $c=c(\left\|\mathrm{R}\right\|_{\infty},m)>0$, $\tilde{c}=\tilde{c}(\left\|\mathrm{R}\right\|_{\infty},m)>0$}, 
\end{align}
and we set $\sigma:= c+1>0$. We are going to prove the existence of a 
$$
C=C(m,p,\left\|\mathrm{R}\right\|_{\infty},\left\|\nabla \mathrm{R}\right\|_{\infty},D,\delta)>0
$$
such that for all $u\in\mathsf{C}^{\infty}_{\Ic}(M)$ one has
\begin{equation}
\left\Vert \nabla\mathrm{d}(\Delta_{0}+\sigma)^{-1}u\right\Vert _{p}\leq
C\left\Vert u\right\Vert _{p},\label{end1}
\end{equation}
which under geodesic completeness is equivalent to
\begin{equation}
\left\Vert \nabla(\Delta_{1}+\sigma)^{-1/2}\mathrm{d}(\Delta_{0}
+\sigma)^{-1/2}u\right\Vert _{p}\leq C\left\Vert u\right\Vert _{p}.\label{end2}
\end{equation}
To this end, we start by observing that by a classical result on
Riesz-transforms of functions by Bakry \cite{bakry} (see also \cite{li,li2} for the weighted case), there is a constant
$C_1=C_1(p)>0$ with
\begin{equation}
\left\Vert \mathrm{d}(\Delta_{0}+\sigma)^{-1/2}\right\Vert _{p,p}\leq
C_1.\label{cou2}
\end{equation}
Next, we are going to use Theorem 4.1 in \cite{thalmaier} in combination with
Example 2.6 therein to treat the $\nabla(\Delta_{1}+\sigma)^{-1/2}$ part: To
this end, let us first note that 
$$
\Delta_{1}=\nabla^{\dagger}\nabla+\mathrm{Ric}({\sharp},{\sharp}),
$$
so that applying Lemma \ref{edh} (where we omit obvious essential self-adjointness arguments) with $S=\nabla$ (on $1$-forms), and $T=\mathrm{Ric}({\sharp},{\sharp})$, which is read as a self-adjoint multiplication operator, bounded by assumption (\ref{gal}), we get that the operator
$$
T_{\sigma}:= \nabla(\Delta_{1}+\sigma)^{-1/2}
$$
from Theorem 4.1 in \cite{thalmaier} is bounded in the $\mathsf{L}^2$-sense, with operator norm $\leq 1$. It remains to check the corresponding assumptions $A$ and $B$ from
\cite{thalmaier}: Here, the validity of assumption $A$ follows immediately from
our curvature assumptions and (\ref{gal}), cf. Example 2.6 from \cite{thalmaier}. Assumption $B1$ follows from the
Laplacian comparison theorem and (\ref{gal}), and assumption $B3$ is implied by the usual
Li-Yau heat kernel estimates, using again (\ref{gal}). Finally, $B2$ is precisely our volume assumption (\ref{B2}). Thus,
by Theorem 4.1 in \cite{thalmaier} we get a
\[
C_2=C_2(m,p,\left\|\mathrm{R}\right\|_{\infty},\left\|\nabla \mathrm{R}\right\|_{\infty},D,\delta)>0
\]
with
\[
\left\Vert \nabla(\Delta_{1}+\sigma)^{-1/2}\right\Vert _{p,p}\leq
C_2 ,
\]
which, in combination with (\ref{cou2}), proves (\ref{end2}) with $C:=C_1C_2$, thus
(\ref{end1}), and the proof is complete.
\end{proof}

\section{Proof of Theorem \ref{counter}}\label{model}

In this section, we construct an explicit example of a complete Riemannian
manifold $M$ with unbounded curvature and that does not support the global
$\IL^{2}$-Calder\'{o}n-Zygmund inequality
\begin{equation}
\left\Vert \mathrm{Hess}\left(  u\right)  \right\Vert _{{2}}\leq C(
\left\Vert \Delta u\right\Vert _{{2}}+\left\Vert u\right\Vert _{{2}
}) \text{, }u\in \C_{c}^{\infty}\left(  M\right)  .\tag*{CZ($2$)}
\end{equation}
Roughly speaking, in order to violate $\mathrm{CZ}(2)$, the idea is to
minimize the contribution of $\Delta u$ with respect to $\mathrm{Hess}\left(
u\right)  $. Clearly, the best way to do this would be to choose $u$ harmonic
(and not affine) but this is impossible because $u$ has compact support.
To overcome the problem, we can take $u$ as the composition of a proper harmonic
function with a singularity at the origin and a cut-off function of
$\mathbb{R}$, compactly supported in $\left(  0,\infty\right)  $. Using this
composition we get rid of the singularity and produce a smooth, compactly
supported function whose $\IL^{2}$-norm of the Laplacian can be small when
compared with that of the Hessian. We shall implement this construction on a model
manifold where, for rotationally symmetric functions, the expressions of the
$\IL^{2}$-norms involved in $\mathrm{CZ}(2)$ are very explicit and directly
related to the geometry of the underlying space.\medskip

By an $m$-dimensional model manifold $\mathbb{R}_{\sigma}^{m}$ we mean the
Euclidean space $\mathbb{R}^{m}$ endowed with the smooth, complete Riemannian
metric that, in polar coordinates, writes as
\[
g=\Id r\otimes \Id r+\sigma^{2}\left(  r\right)  g_{\mathbb{S}^{m-1}},
\]
where $g_{\mathbb{S}^{m-1}}$ is the standard metric of $\mathbb{S}^{m-1}$ and
$\sigma:[0,\infty)\rightarrow\lbrack0,\infty)$ \ is a smooth function
satisfying the following structural conditions:
\[
\begin{array}
[c]{ll}
\text{(a)} & \sigma^{\left(  2k\right)  }\left(  0\right)  =0,\text{ }\forall
k=0,1,...\\
\text{(b)} & \sigma^{\prime}\left(  0\right)  =1\\
\text{(c)} & \sigma\left(  t\right)  >0\text{, }\forall t>0.
\end{array}
\]
We can always identify $\sigma$ with its smooth, odd extension $\sigma
:\mathbb{R}\rightarrow\mathbb{R}$ such that $\sigma\left(  t\right)
=-\sigma\left(  -t\right)  $ for every $t\leq0$. Recall that the sectional
curvatures of $\mathbb{R}_{\sigma}^{m}$ are given by
\begin{align*}
\mathrm{Sec}\left(  X\wedge\nabla r\right)    & =-\frac{\sigma^{\prime\prime
}}{\sigma}\\
\mathrm{Sec}\left(  X\wedge Y\right)    & =\frac{1-(\sigma^{\prime})^{2}
}{\sigma^{2}},
\end{align*}
for every $g$-orthonormal vectors $X,Y\in\nabla r^{\bot}$, where $\nabla r$
represents the radial direction. Moreover, observe that the Riemannian measure of $\mathbb{R}^m_{\sigma}$ is given by 
$$
\Id\mu = \sigma^{m-1}(r)\cdot\Id r \cdot \Id\mu_{\mathbb{S}^{m-1}}, 
$$
where $\Id\mu_{\mathbb{S}^{m-1}}$ denotes the canonical Riemannian measure on $\mathbb{S}^{m-1}$.

\medskip

Let us assume that
\[
\int^{\infty}\frac{\Id t}{\sigma^{m-1}\left(  t\right)  }=\infty.
\]
From the potential theoretic viewpoint, this means that $\mathbb{R}_{\sigma
}^{m}$ is parabolic, namely, the minimal positive Green kernel of the
Laplace-Beltrami operator of $\mathbb{R}_{\sigma}^{m}$ is identically
$\infty$. Then,
\[
G\left(  r\right)  =\int_{1}^{r}\frac{\Id t}{\sigma^{m-1}\left(  t\right)  }
\]
is a smooth, positive, strictly-increasing function on $\left(  0,\infty \right)  $ satisfying
\[
G\left(  r\right)  \left\{
\begin{array}
[c]{ll}
=\infty & \text{if }r=\infty\\
>0 & \text{if }r>1\\
=0 & \text{if }r=1\\
<0 & \text{if }0<r<1\\
=-\infty & \text{if }r=0^{+}.
\end{array}
\right.
\]
Moreover, $G\left(  r\right)  $ gives rise to a smooth, rotationally symmetric
harmonic function $G\left(  x\right)  $ on $\mathbb{R}_{\sigma}^{m}
\backslash\left\{  0\right\}  $. In particular:
\[
\Delta G=G^{\prime\prime}+\left(  m-1\right)  \frac{\sigma^{\prime}}{\sigma
}G^{\prime}=0\text{, on }\mathbb{R}_{\sigma}^{m}\backslash\left\{  0\right\}
.
\]

We need the following computational Lemma.

\begin{Lemma} \label{lemma_counter} 
Let $\mathbb{R}_{\sigma}^{m}$ \ be a complete, parabolic,
model manifold so that $\sigma^{1-m}\notin \IL^{1}\left(  +\infty\right)  $.
Let, as above,
\[
G\left(  r\right)  =\int_{1}^{r}\frac{1}{\sigma^{m-1}\left(  t\right)  }\Id t,
\]
and let $\{\alpha_k\},\{\beta_k\}\subset (0,\infty)$ be two sequences such that $1<\alpha_{k}<\beta_{k}$. Assume further that for any $k$ one has given a function $\phi_{k}\in \C_{c}^{\infty}\left(    0,\infty \right)$ which satisfies $\mathrm{supp}(\phi_{k})\subset [\alpha_{k},\beta_{k}]$, and define $u_{k}\in \C^{\infty}\left(  \mathbb{R}_{\sigma}^{m}\right)$ by setting $u_{k}(x)=\phi_{k}\left(  G(x)\right)$. Then, each $u_{k}$ is in fact compactly supported in 
$$
\left\{\alpha_{k}\leq G\leq\beta_{k}\right\}  \subset \mathbb{R}_{\sigma}^{m},
$$
and one has
\[
\begin{array}
[c]{l}
\left\Vert \mathrm{Hess}\left(  u_{k}\right)  \right\Vert _{{2}}^{2}
\geq\omega_{m}\int_{\alpha_{k}}^{\beta_{k}}\left(  \phi_{k}^{\prime}\left(
s\right)  \right)  ^{2}\left(  \dfrac{\sigma^{\prime}}{\sigma}\left(
G^{-1}\left(  s\right)  \right)  \right)  ^{2}\Id s,\bigskip\\
\left\Vert \Delta u_{k}\right\Vert _{2}^{2}=\omega_{m}\int_{\alpha_{k}
}^{\beta_{k}}\dfrac{\left(  \phi_{k}^{\prime\prime}\left(  s\right)  \right)
^{2}}{\left(  \sigma\left(  G^{-1}\left(  s\right)  \right)  \right)
^{2\left(  m-1\right)  }}\Id s,\bigskip\\
\left\Vert u_{k}\right\Vert _{2}^{2}=\omega_{m}\int_{\alpha_{k}}
^{\beta_{k}}\left(  \phi_{k}\left(  s\right)  \right)  ^{2}\left(
\sigma\left(  G^{-1}\left(  s\right)  \right)  \right)  ^{2\left(  m-1\right)
}\Id s,
\end{array}
\]
where $\omega_{m}>0$ is a dimensional constant.
\end{Lemma}

\begin{proof}
Recall that
\[
\mathrm{Hess}\left(  u_{k}\right)  =u_{k}^{\prime\prime}\cdot \Id r\otimes
\Id r+\frac{\sigma^{\prime}}{\sigma}u_{k}^{\prime}\cdot\sigma^{2}g_{\mathbb{S}
^{m-1}}.
\]
Therefore, we have
\begin{align*}
\left\vert \mathrm{Hess}\left(  u_{k}\right)  \right\vert ^{2} &  =\left(
u_{k}^{\prime\prime}\right)  ^{2}+\left(  m-1\right)  \left(  u_{k}^{\prime
}\right)  ^{2}\left(  \frac{\sigma^{\prime}}{\sigma}\right)  ^{2}\\
&  \geq\left(  u_{k}^{\prime}\right)  ^{2}\left(  \frac{\sigma^{\prime}
}{\sigma}\right)  ^{2}.
\end{align*}
Since
\begin{align*}
u_{k}^{\prime}\left(  r\right)   &  =\phi_{k}^{\prime}\left(  G\right)
G^{\prime}\\
&  =\phi_{k}^{\prime}\left(  G\right)  \frac{1}{\sigma^{m-1}}
\end{align*}
we get
\begin{align*}
\left\vert \mathrm{Hess}\left(  u_{k}\right)  \right\vert ^{2} &  \geq\left(
u_{k}^{\prime}\right)  ^{2}\left(  \frac{\sigma^{\prime}}{\sigma}\right)
^{2}\\
&  =\left(  \phi_{k}^{\prime}\left(  G\right)  \right)  ^{2}\left(  G^{\prime
}\right)  ^{2}\left(  \frac{\sigma^{\prime}}{\sigma}\right)  ^{2}\\
&  =\left(  \phi_{k}^{\prime}\left(  G\right)  \right)  ^{2}\left(
\frac{\sigma^{\prime}}{\sigma}\right)  ^{2}\frac{G^{\prime}}{\sigma^{m-1}}.
\end{align*}
In particular, letting $\omega_{m}$ be the volume of the standard $\left(
m-1\right)  $-sphere,
\begin{align*}
\left\Vert \mathrm{Hess}\left(  u_{k}\right)  \right\Vert _{2}^{2} &
=\omega_{m}\int_{0}^{\infty}\left\vert \mathrm{Hess}\left(  u_{k}\right)
\right\vert ^{2}\sigma^{m-1}\Id t\\
&  \geq\omega_{m}\int_{0}^{\infty}\left(  \phi_{k}^{\prime}\left(  G\right)
\right)  ^{2}\left(  \frac{\sigma^{\prime}}{\sigma}\right)  ^{2}G^{\prime}\Id t\\
&  =\omega_{m}\int_{G^{-1}\left(  \alpha_{k}\right)  }^{G^{-1}\left(
\beta_{k}\right)  }\left(  \phi_{k}^{\prime}\left(  G\right)  \right)
^{2}\left(  \frac{\sigma^{\prime}}{\sigma}\right)  ^{2}G^{\prime}\Id t\\
&  =\omega_{m}\int_{\alpha_{k}}^{\beta_{k}}\left(  \phi_{k}^{\prime}\left(
s\right)  \right)  ^{2}\left(  \frac{\sigma^{\prime}}{\sigma}\left(
G^{-1}\right)  \right)  ^{2}\Id s
\end{align*}
where, in the last equality, we have used the change of variable $G\left(
t\right)  =s$. 
Similarly, on noting that
\begin{align*}
u_{k}^{\prime\prime}\left(  r\right)   &  =\phi_{k}^{\prime\prime}\left(
G\right)  \left(  G^{\prime}\right)  ^{2}+\phi_{k}^{\prime}\left(  G\right)
G^{\prime\prime}\\
&  =\phi_{k}^{\prime\prime}\left(  G\right)  \frac{1}{\sigma^{2\left(
m-1\right)  }}+\phi_{k}^{\prime}\left(  G\right)  G^{\prime\prime},
\end{align*}
using also the harmonicity of $G$, we compute
\begin{align*}
\Delta u_{k}  &  =u_{k}^{\prime\prime}+\left(  m-1\right)  \frac
{\sigma^{\prime}}{\sigma}u_{k}^{\prime}\\
&  =\phi_{k}^{\prime\prime}\left(  G\right)  \frac{1}{\sigma^{2\left(
m-1\right)  }}+\phi_{k}^{\prime}\left(  G\right)  \left( G^{\prime\prime
}+\left(  m-1\right)  \frac{\sigma^{\prime}}{\sigma}G^{\prime}\right) \\
&  =\phi_{k}^{\prime\prime}\left(  G\right)  \frac{1}{\sigma^{2\left(
m-1\right)  }}.
\end{align*}
It follows that
\begin{align*}
\left\Vert \Delta u_{k}\right\Vert _{2}^{2}  &  =\omega_{m}\int
_{0}^{\infty}\left(  \phi^{\prime\prime}_k\left(  G\right)  \right)  ^{2}
\frac{1}{\sigma^{3\left(  m-1\right)  }}\Id t\\
&  =\omega_{m}\int_{G^{-1}\left(  \alpha_{k}\right)  }^{G^{-1}\left(
\beta_{k}\right)  }\left(  \phi_{k}^{\prime\prime}\left(  G\right)  \right)
^{2}\frac{1}{\sigma^{2\left(  m-1\right)  }}G^{\prime}\Id t\\
&  =\omega_{m}\int_{\alpha_{k}}^{\beta_{k}}\left(  \phi_{k}^{\prime\prime
}\left(  s\right)  \right)  ^{2}\frac{1}{\left(  \sigma\left(  G^{-1}\right)
\right)  ^{2\left(  m-1\right)  }}\Id s
\end{align*}
Finally, we compute
\begin{align*}
\left\Vert u_{k}\right\Vert _{2}^{2}  &  =\omega_{m}\int_{0}^{\infty
}\left(  \phi_{k}\left(  G\right)  \right)  ^{2}\sigma^{m-1}\Id t\\
&  =\omega_{m}\int_{G^{-1}\left(  \alpha_{k}\right)  }^{G^{-1}\left(
\beta_{k}\right)  }\left(  \phi_{k}\left(  G\right)  \right)  ^{2}
\sigma^{2\left(  m-1\right)  }G^{\prime}\Id t\\
&  =\omega_{m}\int_{\alpha_{k}}^{\beta_{k}}\left(  \phi_{k}\left(  s\right)
\right)  ^{2}\left(  \sigma\left(  G^{-1}\right)  \right)  ^{2\left(
m-1\right)  }\Id s.
\end{align*}
This completes the proof.
\end{proof}

Now we proceed with the choice of the warping function $\sigma$ and of the
cut--off functions $\phi_{k}$ in such a way that $\mathrm{CZ}(2)$ is violated
along the corresponding sequence of test-functions $u_{k}$. To this end, we
begin by taking
\[
m=2,\text{ }\alpha_{k}=k,\text{ }\beta_{k}=k+1.
\]

Next, we choose $\sigma\left(  t\right)  $ in such a way that
\[
t\leq\sigma\left(  t\right)  \leq t+1,\text{ }t>1.
\]

\begin{Remark}
We explicitly note that, by definition of $G$,
\[
\log\left(  \frac{t+1}{2}\right)  \leq G\left(  t\right)  \leq\log\left(
t\right)  ,\text{ }t>1.
\]
It follows that
\[
\mathrm{e}^{s}\leq G^{-1}\left(  s\right)  \leq2\mathrm{e}^{s}-1,\text{ }s>0.
\]
In particular,
\[
\mathrm{e}^{k}\leq G^{-1}\left(  s\right)  \leq2 \mathrm{e}^{k+1}-1,\text{ on }[k,k+1].
\]
Whence, since
\[
G^{-1}\left(  k+1\right)  -G^{-1}\left(  k\right)  \geq \mathrm{e}^{k+1}-2\mathrm{e}^{k}
+1=\mathrm{e}^{k}\left(  \mathrm{e}-2\right)  +1>1
\]
we also deduce that, for each $k$, there exists some integer $h=h\left(
k\right)  >k$ such that
\[
\left[  h,h+1\right]  \subseteq\left[  G^{-1}\left(  k\right)  ,G^{-1}\left(
k+1\right)  \right]  \text{.}
\]
Furthermore,
\[
\mathrm{e}^{s}\leq\sigma\left(  G^{-1}\left(  s\right)  \right)  \leq2\mathrm{e}^{s}\text{.}
\]
These estimates will be used repeatedly in the sequel.
\end{Remark}

We also require that $\sigma\left(  t\right)  $ oscillates in each interval
$\left[  k,k+1\right]  $ with a slope that increases with $k$. We can model
the oscillating part by segments like 
$$
t\mapsto\bar{t}+\left(  t-\bar
{t}\right)  \left(  \varepsilon_{k}+1\right)  /\varepsilon_{k}\>\text{ on $[\bar
{t},\bar{t}+\varepsilon_{k}]$,} 
$$
and 
$$
t\mapsto\bar{t}+2\varepsilon
_{k}+\left(  \bar{t}+2\varepsilon_{k}-t\right)  /\varepsilon_{k}\>\text{ on $[\bar
{t}+\varepsilon_{k},\bar{t}+2\varepsilon_{k}]$}, 
$$
with $\varepsilon_{k}\to 
 0+$. The upper and lower angles are smoothened out in regions as close to
the vertices as we desire. The smoothing can be realized via concave (resp.
convex) functions; see \cite{ghomi}.

\begin{Remark}
By construction, each rectilinear portion of $(\sigma^{\prime})^{2}$ grows
like $1/\varepsilon_{k}^{2}$ on an interval of approximate length
$\varepsilon_{k}$.
\end{Remark}

\begin{Remark}
The smoothing is obtained via functions of increasingly high second
derivative. It follows that the Gaussian curvature of $\mathbb{R}_{\sigma}
^{2}$ explodes to  $\infty$ as the oscillatory part becomes closer and closer
to vertical segments. We also point out that, in dimensions
$m\geq3$ this construction gives rise to a model manifold whose sectional
curvatures, both radial and tangential, explode to $\infty$. Finally, observe
that, due to the profile of $\sigma$, the manifold should have vanishing injectivity
radius (although, at the pole $0\in\mathbb{R}_{\sigma}^{2}$, it holds that
$r_{\mathrm{inj}}\left(  0\right)  =\infty$).
\end{Remark}

To conclude, we choose $\phi_{k}\left(  t\right)  =\phi\left(  t-k\right)  $
where $\phi\left(  t\right)  $ is compactly supported in $[0,1]$ and satisfies
$\phi\left(  t\right)  =2t$ on $[1/4,1/2].$ In this way, $\phi_{k}^{\prime
}\equiv2$ on the interval of fixed length $[k+1/4,k+1/2]$ so to capture some
of the oscillations of $\sigma$. Note also that $\left\Vert \phi
_{k}\right\Vert _{\infty}$ and $\left\Vert \phi_{k}^{\prime\prime
}\right\Vert _{\infty}$ are uniformly bounded.\medskip

Now, according to Lemma \ref{lemma_counter}, we have the following estimates:
\[
\left\Vert u_{k}\right\Vert _{2}^{2} =\omega_{2}\int_{k}^{k+1}\left(
\phi_{k}\left(  s\right)  \right)  ^{2}\left(  \sigma\left(  G^{-1}\left(
s\right)  \right)  \right)  ^{2}\Id s \leq C\mathrm{e}^{2k},
\]
and
\begin{align*}
\left\Vert \Delta u_{k}\right\Vert _{2}^{2} &= \omega_{2}\int_{k}^{k+1}
\dfrac{\left(  \phi_{k}^{\prime\prime}\left(  s\right)  \right)  ^{2}}{\left(
\sigma\left(  G^{-1}\left(  s\right)  \right)  \right)  ^{2}}\Id s\\
&\leq C\int
_{k}^{k+1}\frac{\Id s}{\mathrm{e}^{2s}}=\frac
{C}{\mathrm{e}^{2k}}
\end{align*}
and, finally,
\begin{align*}
\left\Vert \mathrm{Hess}\left(  u_{k}\right)  \right\Vert _{2}^{2} &
\geq\omega_{2}\int_{k}^{k+1}\left(  \phi_{k}^{\prime}\left(  s\right)
\right)  ^{2}\left(  \dfrac{\sigma^{\prime}}{\sigma}\left(  G^{-1}\left(
s\right)  \right)  \right)  ^{2}\Id s\\
&  \geq\frac{\omega_{2}}{\mathrm{e}^{2\left(  k+1\right)  }}\int_{k}^{k+1}\left(
\phi_{k}^{\prime}\left(  s\right)  \right)  ^{2}\left(  \sigma^{\prime}\left(
G^{-1}\left(  s\right)  \right)  \right)  ^{2}\Id s\\
&  \geq\frac{C}{\mathrm{e}^{2\left(  k+1\right)  }}\int_{0}^{\varepsilon_{h}}\frac
{1}{\varepsilon_{h}^{2}} \Id s\text{, }\>h=h\left(  k\right)  >k,\\
&  \geq\frac{C}{\mathrm{e}^{2k}}\frac{1}{\varepsilon_{k}}.
\end{align*}
Whence, we deduce that we can choose $\varepsilon_{k}\searrow0$ in such a way
that $\mathrm{CZ}(2)$  is violated.
This completes the proof of Theorem \ref{counter}.

\section{Calder\'on-Zygmund inequalities on $H$-hypersurfaces}\label{etre2}

Let $\mathbb{M}^{m+1}(c)$ denote the complete, simply connected space-form of constant sectional curvature $c \leq 0$. In this section we explore the validity of the $\IL^p$-Calder\'on-Zygmund inequalities on a largely investigated class of submanifolds of $\mathbb{M}^{m+1}(c)$: the hypersurfaces of constant mean curvature $H \in \mathbb{R}$ ($H$-hypersurfaces for short) with finite total scalar curvature.

Let $f: M \to \mathbb{M}^{m+1}(c)$ be a complete, connected, oriented, isometrically immersed submanifold of dimension $\dim M=m\geq 3$. Its second fundamental tensor, with respect to a chosen Gauss map $\nu$, is denoted by $\mathbf{II}$. The corresponding mean curvature vector field is
$\mathbf{H}=\mathrm{trace}(\mathbf{II})/m$. We write $\mathbf{H}=H\nu$, where the smooth function $H$ is the mean curvature function of the hypersurface, and we assume that $H$ is constant. The \emph{total curvature} of the constant mean curvature hypersurface $M$ is the $\IL^{m}$-norm of its traceless second fundamental tensor $\mathbf{\Phi}=\mathbf{II}-\mathbf{H}g$. We say that $M$ has \emph{finite total curvature} if $\|\mathbf{\Phi}\|_m  < \infty.$ In case $\mathbf{H}=0$ the hypersurface is called \emph{minimal} and the finite total curvature condition reduces to $\|\mathbf{II}\|_m <\infty$.

A complete, oriented $H$-hypersurface $M$ in $\mathbb{M}^{m+1}(c)$ of finite total curvature must be necessarily closed provided $H^2+c>0$. Indeed, according to \cite{Anderson, BDS}, the traceless tensor $\mathbf{\Phi}$ satisfies the decay condition
\begin{equation}\label{curvest}
\sup_{M\setminus \mathrm{B}^M_R}|\mathbf{\Phi}| \to 0 \textit{ as } R \to \infty.
\end{equation}
See also \cite{PV}. Therefore, by Gauss equations, the Ricci curvature of $M$ is positively pinched outside a compact set, \cite{Leung}, and the compactness conclusion follows from the Bonnet-Myers type theorems in \cite{Galloway}.

Since, on the one hand, the condition $H^2+c > 0$ implies obvious non-existence results and, on the other hand, we are mainly interested in non-compact situations, from now on we assume that
\begin{equation}\label{H-compatibility}
H^2+c \leq 0.
\end{equation}
In particular, if $c=0$, then $M$ is minimal.

Under the compatibility condition (\ref{H-compatibility}), it is a well known consequence of the curvature estimate (\ref{curvest}) that the $H$-hypersurface is properly immersed and it has a finite number of ends, each of which is diffeomorphic to a cylinder over some compact hypersurface; \cite{Anderson, Castillon}. Moreover, up to imposing a more stringent pinching on $H$ when $c<0$, the volume of each end is subjected to a certain growth; \cite{Anderson, PV}. Actually, it is known that \emph{any} complete Riemannian manifold isometrically immersed with \emph{bounded mean curvature} into a Cartan-Hadamard manifold satisfies the \emph{non-collapsing} condition at infinity

\begin{equation}\label{non-collapsing}
\inf_{x \in M} \mu(\mathrm{B}_1(x)) = v > 0.
\end{equation}
Indeed, according to \cite{HS}, such a submanifold enjoys the $\IL^1$-Sobolev inequality
\[
\left\|u\right\|_{\frac{m}{m-1}} \leq C( \left\|\grad(u)\right\|_1 + \left\|u\right\|_1),
\]
for every $u\in \C^{\infty}_\Ic(M)$ and for some constant $C>0$ depending on $m$ and $\|\mathbf{H}\|_{\infty}<+\infty$. Whence, it is standard to deduce the validity of (\ref{non-collapsing}) by integrating the differential inequality 
$$
\mu(\mathrm{B}_r(x))^{\frac{m-1}{m}} \leq C \left(\frac{\Id}{\Id r}\mu(\mathrm{B}_r(x))+\mu(\mathrm{B}_r(x))\right)
$$
that arises from a suitable choice of the (radial) cut-off functions $u$ and a standard application of the co-area formula. This is part of the classical Federer-Fleming argument. Note that,  using a rescaling procedure, the unit ball in (\ref{non-collapsing}) can be replaced by any ball of fixed radius $r>0$. Obviously, in this case, the constant $v$ will depend on $r$.

We are now ready to prove the main result of the section.

\begin{Theorem}\label{th_hypersurfaces}
Let $f:M \to \mathbb{M}^{m+1}(c)$ be a complete, non-compact, oriented, $H$-hypersurface with finite total curvature into the complete, simply connected space-form $\mathbb{M}^{m+1}(c)$ of constant  sectional curvature $c \leq 0$. Then $M$ satisfies the assumptions of Theorem \ref{centrale1}, in particular, for every $1<p< \infty$, the Calder\'on-Zygmund inequality $\mathrm{CZ}(p)$ holds on $M$.
\end{Theorem} 

\begin{proof}
Combining the Gauss equations with estimate (\ref{curvest}) on the traceless second fundamental tensor, we deduce that $M$ has bounded sectional curvature. On the other hand, $M$ satisfies the non-collapsing condition (\ref{non-collapsing}). It follows from  Theorem 4.7 in \cite{CGT} that $r_\mathrm{inj}(M)>0$. Therefore, we can apply Theorem \ref{centrale1} above and conclude the validity of $\mathrm{CZ}(p)$.
\end{proof}

\begin{Remark}
The decay of the traceless second fundamental tensor of the $H$-hypersurface $M$ of $\mathbb{M}^{m+1}(c)$ holds provided $M$ has finite $\IL^{p}$-total curvature $\|\mathbf{\Phi}\|_p < \infty$ for some $m \leq p < \infty$, \cite{PV}. The conclusion of Theorem \ref{th_hypersurfaces} can be extended accordingly. 
\end{Remark}

\begin{Remark}
As a matter of fact, inspection of the proof of Theorem \ref{th_hypersurfaces} shows that it relies on two facts: (a) by Gauss equations, the sectional curvature of a manifold $M$ is bounded if $M$ is isometrically immersed, with bounded second fundamental form, into an ambient  manifold of bounded curvature; (b) the injectivity radius of $M$ is bounded from below by a positive constant provided $M$ is isometrically immersed, with bounded mean curvature, into a Cartan-Hadamard manifold. Therefore, Theorem \ref{th_hypersurfaces} can be extended in the following more abstract form:  
\end{Remark}

\begin{Theorem}
Let $f: M \to N$ be a complete Riemannian manifold isometrically immersed into a complete, simply connected manifold $N$ with sectional curvatures satisfying $-A^2 \leq \mathrm{Sec}_{N} \leq 0$. If the second fundamental tensor of the immersion satisfies $\|\mathbf{II}\|_{\infty}<\infty$, then $M$ satisfies the assumptions of Theorem \ref{centrale1}, in particular $\mathrm{CZ}(p)$ holds on $M$, for every $1<p<\infty$.
\end{Theorem}

\appendix
\numberwithin{equation}{section}
\counterwithin{theorem}{section}

\section{Harmonic coordinates}\label{harm}

In this section, we collect some facts concerning harmonic coordinates. Let again $M\equiv (M,g)$ be an arbitrary smooth Riemannian $m$-manifold without boundary, let $\nabla$ be the Levi-Civita connection and $\Delta $ the Laplace-Beltrami operator. 

\begin{Definition} Let $x\in M$, $Q\in (1,\infty)$, $k\in\IN_{\geq 0}$, $\alpha\in (0,1)$. The \emph{$\mathsf{C}^{k,\alpha}$-harmonic radius of $M $ with accuracy $Q$ at $x$} is defined to be the largest real number $r_{Q,k,\alpha}(x)$ with the following property: The ball $\mathrm{B}_{r_{Q,k,\alpha}(x)}(x)$ admits a centered harmonic coordinate system 
$$
\phi: \mathrm{B}_{r_{Q,k,\alpha}(x)}(x)\longrightarrow  \IR^m,
$$ 
(that is, $\phi(x)=0$ and $ \Delta\phi^j=0$ on $\mathrm{B}_{r_{Q,k,\alpha}(x)}(x)$ for each $j$), such that 
\begin{align}
Q^{-1}(\delta_{ij})\leq (g_{ij})\leq Q(\delta_{ij})\text{ in $\mathrm{B}_{r_{Q,k,\alpha}(x)}(x)$ as symmetric bilinear forms},\label{hari1}
\end{align}
 and for all $i,j\in\{1,\dots, m\}$,
\begin{align}\label{hari2}
 &\sum_{\beta\in\IN^m, 1\leq |\beta|\leq k}r_{Q,k,\alpha}(x)^{|\beta|} \sup_{x\rq{}\in\mathrm{B}_{r_{Q,k,\alpha}(x)}(x) } |\partial_{\beta}g_{ij}(x\rq{})|\\
&+\sum_{\beta\in\IN^m, |\beta|=k}r_{Q,k,\alpha}(x)^{k+\alpha} \sup_{x\rq{},x\rq{}\rq{}\in\mathrm{B}_{r_{Q,k,\alpha}(x)}(x), x\rq{}\rq{}\ne x\rq{} }\f{ |\partial_{\beta}g_{ij}(x\rq{})-\partial_{\beta}g_{ij}(x\rq{}\rq{})|}{\Id(x\rq{},x\rq{}\rq{})^{\alpha}}\nn\\
&\leq Q-1.\nn
\end{align}

We shall refer to a coordinate system as above as a \emph{$\mathsf{C}^{k,\alpha}$-harmonic coordinate system} with accuracy $Q$ on $\mathrm{B}_{r_{Q,k,\alpha}(x)}(x)$.
\end{Definition}

\begin{Remark} 1. Note that, when compared with the corresponding definition from \cite{hebey}, we additionally require $\phi(x)= 0$ here.\\
2. It is easily checked that the function $x\mapsto r_{Q,k,\alpha}(x)$ is globally Lipschitz.\\
3. By polarization, the inequality (\ref{hari1}) implies that for some $C=C(m,Q)>0$ it holds that 
$$
\max_{i,j\in\{1,\dots,m\}}\sup_{x\rq{}\in \mathrm{B}_{r_{Q,k,\alpha}(x)}(x)} |g_{ij}(x\rq{})|\leq C,
$$
in particular, putting everything together, there is a continuous decreasing function 
$$
F=F_{Q,k,\alpha,m}:(0,\infty)\longrightarrow  (0,\infty), 
$$
such that the Euclidean $\mathsf{C}^{k,\alpha}$-norm of the metric in this coordinates satisfies
\begin{equation}\label{Cka-norm}
\max_{i,j\in\{1,\dots,m\}}\left\|g_{ij}\right\|_{\mathsf{C}^{k,\alpha}}\leq  F(r_{Q,k,\alpha}(x)).
\end{equation}
This justifies the name \emph{$\mathsf{C}^{k,\alpha}$-(harmonic) coordinate system}.\\
\end{Remark}

\begin{Remark}\label{rem_inverse}
The natural differential operators of $M$ (such as the gradient, the Laplacian
and the Hessian of a given function) are defined in terms of the inverse
metric coefficients $g^{ij}$. It is easy to see that, within the coordinate
ball $\mathrm{B}_{r_{Q,k,\alpha}\left(  x\right)  }\left(  x\right)  $, they inherit
the $\C^{k,\alpha}$-type control, in terms of $Q,m,\alpha,k$, of the metric coefficients $g_{ij}$. Indeed,
the Cramer formula states that
\begin{equation}
g^{ij}=\left(  -1\right)  ^{i+j}\frac{\det G_{ji}}{\det\left(  g_{ij}\right)
},\label{inverse1}
\end{equation}
where $G_{ij}$ denotes the $\left(  m-1\right)  \times\left(  m-1\right)  $
matrix obtained from $(g_{ij})$ by deleting the $i^{\text{th}}$-row and the
$j^{\text{th}}$-column. Both the numerator and the denominator of
(\ref{inverse1}) are obtained as the sum of products of $C^{0,\alpha}
$-controlled functions and, by (\ref{hari1}), $Q^{-m}\leq\det(g_{ij})\leq
Q^{m}$.
Therefore, we can obtain a $\C^{0,\alpha}$ control of the functions
$g^{ij}$ by using the $\C^{0,\alpha}$ estimates of $g_{ij}$ in combination with
the following elementary fact: \bigskip

Assume that $f,h:U\subseteq\mathbb{R}^{m}\rightarrow\mathbb{R}$ satisfy
$C^{-1}\leq h\leq C$ and $|f| \leq D$, for some constants $C,D>0$. Then:
\begin{align*}
\left\vert \Delta_{x,y}(f/h)  \right\vert  & \leq
C^{3}\left\vert \Delta_{x,y}(f)\right\vert +C^{2}D\left\vert \Delta
_{x,y}(h)\right\vert \\
\left\vert \Delta_{x,y}\left(  fh\right)  \right\vert  & \leq C\left\vert
\Delta_{x,y}(f)\right\vert +D\left\vert \Delta_{x,y}(h)\right\vert \\
\left\vert \Delta_{x,y}\left(  f+h\right)  \right\vert  & \leq\left\vert
\Delta_{x,y}(f)\right\vert +\left\vert \Delta_{x,y}(h)\right\vert ,
\end{align*}
where, to simplify the writings, we have set $\Delta_{x,y}(\bullet)=\bullet(
y)  -\bullet\left(  x\right)  $. \bigskip

Now,  differentiating the identity $g^{ik}\cdot g_{kj}=\delta_{ij}$ we get
\begin{equation}\label{firstder^ij}
\partial_t g^{ij}=-g^{ih}\cdot \partial_t g_{hj} \cdot g^{jk}.
\end{equation}
It follows that a  $\C^{0,\alpha}$ control of $\partial g^{ij}$ is obtained from those of $g^{ij}$ and $\partial g_{ij}$. Proceeding inductively on the derivatives of (\ref{firstder^ij}) we finally deduce the claimed $\C^{k,\alpha}$ estimate of $g^{ij}$.
\end{Remark}

The main result in this context states that control on the Ricci curvature up to order $k$ together with control on the injectivity radius imply control on $r_{Q,k+1,\alpha}(x)$. To this end, for any $\Omega\subset M$ open and any $\varepsilon>0$ let
$$
\Omega_{\varepsilon}:=\{x| \>x\in M, \Id(x,\Omega)<\varepsilon\}\subset M
$$
be the $\varepsilon$-neighborhood of $\Omega$. 

\begin{Theorem}\label{univ} Let $Q\in (1,\infty)$, $\alpha\in (0,1)$. Assume that there is an open subset $\Omega\subset M$, and numbers  $k\in\IN_{\geq 0}$, $\varepsilon>0$, $r>0$, $c_0,\dots, c_k>0$ with
$$
\left|\nabla^j \mathrm{Ric}(x)\right|_x\leq c_j,\> r_{\mathrm{inj}}(x)\geq r \text{ for all $x\in \Omega_{\varepsilon}$, $j\in \{0,\dots,k\}$.}
$$
Then there is a constant 
$$
C=C(m,Q,k,\alpha,\varepsilon,r,c_1,\dots,c_k)>0,
$$
such that for all $x\in\Omega$ one has $r_{Q,k+1,\alpha}(x)\geq C$.
\end{Theorem}

\begin{proof}
Except the additional assumption $\phi(x)= 0$ that we have made for harmonic coordinates, this result can be found in \cite{hebey} and the references therein (cf. Theorem 1.3 therein). However, since translations do not effect the required estimates, this is not a restriction.
\end{proof}

In particular, the latter result implies $r_{Q,j,\alpha}(x)>0$ for all $x\in M$, a fact which a priori is not clear at all. One calls the number
$$
r_{Q,j,\alpha}(M):=\inf_{x\in M}r_{Q,j,\alpha}(x)
$$
the \emph{$\mathsf{C}^{k,\alpha}$-harmonic radius for the accuracy $Q$.}

\section{Gluing Riemannian manifolds}\label{gluing}

Suppose we are given two Riemannian manifolds $(M_{1},g_{1})$ and
$(M_{2},g_{2})$ with compact diffeomorphic boundaries and let $f:\partial
M_{1}\rightarrow\partial M_{2}$ be a fixed diffeomorphism\footnote{different
choices of $f$ could produce non-diffeomorphic gluings as the exotic twisted
spheres show.}. The \emph{Riemannian gluing} $\mathcal{M}=M_{1}\cup_{f}M_{2}$
of $M_{1}$ and $M_{2}$ along $f$ is the Riemannian manifold $(\mathcal{M},g)$
defined as follows. \smallskip

As a topological manifold, $\mathcal{M}$ is the quotient space obtained from
the disjoint union $M_{1}\sqcup M_{2}$ under the identification $x\sim f(x)$,
for every $x\in\partial M_{1}$. It turns out that the natural inclusions
$\mathrm{i}_{j} : M_{j} \hookrightarrow\mathcal{M}$, $j=1,2$, are continuous embeddings.

Next, having fixed arbitrarily small collar neighborhoods $\alpha_{j}:\partial
M_{j}\times\lbrack0,2)\rightarrow M_{j}$ of $\partial M_{j}$, $j=1,2$, we
consider the homeomorphism $\alpha:\partial M_{1}\times(-2,2)\rightarrow
\mathcal{M}$ onto a neighborhood $\mathcal{V}$ of $\mathrm{i}_{1}(\partial
M_{1})=\mathrm{i}_{2}(\partial M_{2})$ defined as follows:
\[
\alpha\left(  x,t\right)  =\left\{
\begin{array}
[c]{ll}
\mathrm{i}_{1}\circ\alpha_{1}\left(  x,-t\right)   & t\leq0\\
\mathrm{i}_{2}\circ\alpha_{2}\left(  f\left(  x\right)  ,t\right)   & t\geq0.
\end{array}
\right.
\]
The original differentiable structures on $M_{1}$ and $M_{2}$ are then
extended to a (unique up to diffeomorphisms) differentiable structure on
$\mathcal{M}$ by requiring that the natural inclusions $\mathrm{i}_{j}$,
$j=1,2$, are smooth embeddings and by pretending that $\alpha$ is a smooth
diffeomorphism. See e.g. Chapter 8 of \cite{Hirsch} and Chapter 5 of
\cite{Munkres}.

Finally, let $\mathcal{W}=\alpha(\partial M_{1}\times(-1,1))$, and fix any
Riemannian metric $g_{3}$ on $\mathcal{W}$. For instance, we can pull-back
on $\mathcal{W}$ via $\alpha^{-1}$ a product metric $h+dt\otimes dt$ on $\partial
M_{1}\times(-1,1)$. Let $\xi_{1},\xi_{2}$ and $\xi$ denote a partition of
unity subordinated to the open covering $\mathrm{i}_{1}(M_{1}\backslash
\partial M_{1})$, $\mathrm{i}_{2}(M_{2}\backslash\partial M_{2})$, and $\mathcal{W}$ of
$\mathcal{M}$. A Riemannian metric on $\mathcal{M}$ is defined by setting
\[
g=\xi_{1}\cdot(\mathrm{i}^{-1}_{1})^{\ast}g_{1}+\xi_{2}\cdot(\mathrm{i}^{-1}_{2})^{\ast
}g_{2}+\xi\cdot g_{3}.
\]
Note that, outside the compact neighborhood $\overline{\mathcal{W}}$ of
$\mathrm{i}_{1}(\partial M_{1})=\mathrm{i}_{2}(\partial M_{2})$, $\mathcal{M}$
is isometric to the original open manifolds $M_{j}\backslash\alpha_{j}\left(
\partial M_{j}\times\lbrack0,1]\right)  $. In particular, different choices of
$g_{3}$ will leave the corresponding Riemannian structure of $\mathcal{M}$ in
the same bilipschitz class. Moreover, if $\Omega$ is a domain compactly
contained, e.g., in $M_{1}\setminus\partial M_{1}$, then the collar
neighborhood $\alpha_{1}(\partial M_{1}\times\lbrack0,2))$ of $\partial M_{1}$
can be chosen so to have empty intersection with $\Omega$. Therefore, the
neighborhood $\mathcal{V}\supset\overline{\mathcal{W}}$ of $\mathrm{i}
_{1}(M_{1})=\mathrm{i}_{2}(M_{2})$ does not intersect $\mathrm{i}_{1}(\Omega
)$. Whence, it follows that $\Omega$ can be identified with its isometric copy
$\mathrm{i}_{1}(\Omega)$ into the glued space. \smallskip

Now, if we specialize this construction to the case where $M_{1}=M_{2}$ and
$h=\mathrm{id}$ we obtain \textquotedblleft the\textquotedblright
\ \emph{Riemannian double} $\mathcal{M}=\mathcal{D}\left(  M_{1}\right)  $ of
$M_{1}$. On the other hand, if $M_{1}$ and $M_{2}$ are obtained by delating a
disk from the manifolds without boundaries $\overline{M_{1}}$ and
$\overline{M_{2}}$ then we get the (rough and un-oriented) \emph{Riemannian
connected sum }$\overline{M_{1}}\#\overline{M_{2}}$.

\begin{Acknowledgments}
The authors are indebted to Baptiste Devyver for pointing out the multiplicative inequality contained in Theorem 4.1 of \cite{CD}. Furthermore, the auhors would like to thank Xiang-Dong Li and Feng-Yu Wang for a very helpful correspondence on the literature concerning Riesz-transforms.  
\end{Acknowledgments}

\end{document}